\documentclass[11pt,letterpaper,reqno]{amsart}
\ifdefined\pdfinfoomitdate\pdfinfoomitdate=1\fi
\ifdefined\pdftrailerid\pdftrailerid{}\fi
\ifdefined\pdfsuppressptexinfo\pdfsuppressptexinfo=-1\fi
\usepackage[T1]{fontenc}
\usepackage[utf8]{inputenc}
\usepackage{lmodern}
\usepackage{microtype}
\usepackage{amsmath,amssymb}
\usepackage{mathtools}
\usepackage[letterpaper,margin=1in]{geometry}
\usepackage{graphicx}
\usepackage[hidelinks,unicode]{hyperref}
\makeatletter
\renewcommand\subsection{\@startsection{subsection}{2}%
  \z@{.5\linespacing\@plus.7\linespacing}{.25\linespacing}%
  {\normalfont\bfseries}}
\makeatother

\theoremstyle{plain}
\newtheorem{theorem}{Theorem}[section]
\newtheorem{lemma}[theorem]{Lemma}
\newtheorem{proposition}[theorem]{Proposition}

\newtheorem{conjecture}{Conjecture}

\theoremstyle{remark}
\newtheorem{remark}[theorem]{Remark}

\newcommand{\C}{\mathbb C}
\newcommand{\R}{\mathbb R}
\newcommand{\D}{\mathbb D}
\newcommand{\PP}{\mathbb P}
\newcommand{\cA}{\mathcal A}
\newcommand{\cO}{\mathcal O}
\newcommand{\Ric}{\operatorname{Ric}}
\newcommand{\tr}{\operatorname{tr}}
\newcommand{\Id}{\operatorname{Id}}
\newcommand{\dist}{\operatorname{dist}}
\newcommand{\Area}{\operatorname{Area}}
\newcommand{\Acrit}{A_*}
\newcommand{\HS}{\mathrm{HS}}
\newcommand{\op}{\mathrm{op}}
\newcommand{\norm}[1]{\left\lVert#1\right\rVert}
\newcommand{\abs}[1]{\left\lvert#1\right\rvert}

\title{Analytic Construction of Rational Curves\\
on Fano Manifolds}
\author{Yun-Heng Du}
\author{Bin Guo}
\author{Song-Yan Xie}
\date{\today}

\hypersetup{
  pdftitle={Analytic Construction of Rational Curves on Fano Manifolds},
  pdfauthor={Yun-Heng Du, Bin Guo, Song-Yan Xie}
}

\begin{document}

\begin{abstract}
Inspired by constructions of entire curves in Oka geometry, we construct
rational curves on complex Fano manifolds analytically, by alternately
deforming a holomorphic disk to reduce its area and enlarging its
source. Positive Ricci curvature yields an area-decreasing deformation,
while an affine-lift perturbation enlarges the source at a controlled
area cost. The two operations change the area and the weighted
derivative by amounts we estimate explicitly, and analytic compactness
passes from the resulting disks to a holomorphic sphere. The main result
produces a sphere that meets a prescribed compact fiber without being
contained in it, with an explicit area bound.
\end{abstract}
\maketitle
\pagestyle{plain}
\setcounter{tocdepth}{1}
\tableofcontents

\section{Introduction and main results}
\label{sec:introduction}

\subsection{The analytic construction problem}

We construct rational curves on complex Fano manifolds by
explicit analytic operations, with control of the intermediate disks.
A Fano manifold $X$ is connected, smooth, and compact, with ample
anticanonical line bundle $K_X^{-1}$.

Rational curves are central to Mori's solution~\cite{Mori1979} of
Hartshorne's conjecture on ample tangent bundles. In his monograph
\emph{Ample subvarieties of algebraic varieties}, Hartshorne~\cite{Hartshorne1970} asked whether a smooth complex projective variety
with ample tangent bundle must be projective space. Mori first reduces
the problem to characteristic $p>0$, where the Frobenius map is the
key tool in his construction of rational curves. He then shows that
rational curves also exist on the original complex manifold.

On the analytic side, Siu and Yau~\cite{SiuYau1980} used
energy-minimizing harmonic spheres in their 1980 proof of the Frankel
conjecture. Their argument uses positive holomorphic bisectional
curvature, a stronger condition than the Ricci positivity that Yau's
prescribed-Ricci theorem provides on every Fano manifold~\cite{Yau1978}.
In 1982, Yau asked whether any two points of a compact K\"ahler manifold
with positive Ricci curvature could be joined by a chain of rational curves~\cite[p.~44]{Yau1982Survey}. Rational connectedness of Fano manifolds
was subsequently established by Campana and by Koll\'ar, Miyaoka, and
Mori \cite{Campana1992,KMMFano1992}. An analytic construction remained
a separate challenge, which Siu later proposed to approach through
complex Monge--Amp\`ere equations, without reduction to positive
characteristic~\cite{Siu2009}.

Our approach is based on methods for constructing entire
curves in Oka geometry, in particular those of Xie~\cite{Xie2024Nevanlinna}
(see also~\cite{ChenFornaessXie2025,HuynhXie2021,WuXie2025}). These methods build entire curves
by successively modifying holomorphic disks and enlarging their domains.
Guo and Xie~\cite[Section~5]{GuoXie2024} suggested pursuing such
constructions on Fano manifolds with a slow-growth condition, with the
aim of producing both entire and rational curves analytically.
The construction below realizes this idea in the Fano setting.

A nonconstant entire curve of finite area gives a rational curve by
removal of the singularity at infinity (Theorem~\ref{thm:removal}).
We therefore try to construct such an entire curve as a limit of
holomorphic disks with radii tending to infinity. Brody's lemma in
complex hyperbolicity~\cite{Brody1978} suggests normalizing the derivative at the origin and
obtaining convergence from derivative bounds on compact subsets. In our
setting the total area must also remain uniformly bounded.
Positive Ricci curvature provides the area-decreasing deformations
used for this purpose. Derivative concentration may prevent direct
convergence, so the spheres arising through bubbling are analyzed
separately, retaining their required incidence properties.

To carry out these deformations, we need room to vary holomorphic disks
in the target. In the Fano setting, the following folklore conjecture
in Oka geometry points to such flexibility; see~\cite[p.~660]{Forstneric2025HPrinciple}.

\begin{conjecture}
\label{conj:fano-oka}
Every Fano manifold is Oka.
\end{conjecture}

Our recent theorem that rationally connected smooth projective manifolds
are Oka--1~\cite[Corollary~1.8]{DuGuoXieOka1} gives a first step toward
the conjecture, since it includes every Fano manifold. The successive
approximation and enlargement of source domains used in that work
suggest that there is considerable flexibility in deforming holomorphic
disks. This viewpoint motivates the alternating deformation and
enlargement procedure developed below. Although the Oka--1 theorem
assumes rational connectedness, it is used here only as motivation and
plays no role in the proof. Known cases of the Oka conjecture include smooth
quadric~\cite{Grauert1958} and cubic~\cite{KalimanZaidenberg2025} Fano
hypersurfaces. Guo and Xie have recently proved that every smooth quartic
hypersurface of dimension at least three in projective space is
Oka~\cite{GuoXieQuartic2026};
see also~\cite{XieZhaoOkaK3}.
For background on Oka and Oka--1 theory, we refer to
\cite{Forstneric2017Book,AlarconForstneric2025}.

The construction is reminiscent of gathering and sealing a baozi: the disks
are alternately deformed and enlarged, while the uniform area bound eventually
allows the limiting entire curve to be compactified at infinity.

Area and location in the limit require separate treatment, since
holomorphic approximation alone supplies no uniform area bound as the
source radii tend to infinity, while classical Brody reparametrization
moves the source centers~\cite{Brody1978}: even disks through a fixed
point need not yield an entire curve through that point. Duval
addressed this loss of location by replacing derivative selection with
control of where area accumulates~\cite{Duval2008}; see
also~\cite{Duval2021}. For a compact set charged by an Ahlfors current,
his theorem produces an entire curve whose intersection with that set
has positive area. That localization viewpoint also guides the
construction below, but in our finite-area setting we retain position
through a base constraint and two marked values instead of invoking
the Ahlfors-current theorem itself.

Even when the source center is fixed, bubbling can separate its
prescribed local data from the nonconstant sphere components; the
analysis of bubbling goes back to Sacks and Uhlenbeck's work on
harmonic maps~\cite{SacksUhlenbeck1981} and Gromov's theory of
pseudoholomorphic curves~\cite{Gromov1985}. Since area may also
concentrate near the source boundary, we must obtain a nonconstant
sphere that satisfies both the required incidence conditions and the
area bound.

\subsection{Main results}

For rational connectedness, we need a rational curve that meets a fiber
of the rational quotient without remaining inside it. The main theorem
therefore treats a compact fiber of a holomorphic map $q$ defined on
an open neighborhood $N$. Here $F=q^{-1}(y_0)$ denotes the fiber of $q$
over $y_0$ in $N$, including all its components. We use the Ricci and
area conventions in \eqref{eq:normalization}.

\begin{theorem}
\label{thm:bounded-fiber-escape}
Let $X$ be a complex Fano manifold of dimension $n$, equipped with a
K\"ahler metric $\omega$ satisfying
\[
\Ric(\omega)\geq\kappa\omega,
\qquad \kappa>0.
\]
For an open subset $N\subset X$ and a complex manifold $Y$ of positive
dimension, consider a holomorphic map $q:N\to Y$ and a point $y_0\in Y$.
Suppose that the fiber
$F\coloneqq q^{-1}(y_0)$ is nonempty and compact and that $q$ is a submersion on a
neighborhood of $F$ in $N$. Then there is a nonconstant
holomorphic map $u:\PP^1\to X$ such that
\[
 u(\PP^1)\cap F\neq\varnothing,
 \qquad u(\PP^1)\not\subset F,
 \qquad
 \int_{\PP^1}u^*\omega\leq\frac{2\pi(n+1)}{\kappa}.
\]
\end{theorem}

Every Fano manifold admits such a K\"ahler metric by Yau's
prescribed-Ricci theorem~\cite{Yau1978}, as recalled in
Proposition~\ref{prop:fano-metric}; the metric need not be
K\"ahler--Einstein. Since $X$ is projective, the reduced image of the
sphere in the theorem is a rational curve.

This is an explicit metric area bound depending on the chosen Ricci
lower bound $\kappa$.

The relative viewpoint has a classical antecedent in Campana's
generic-fiber escape criterion and his Fano theorem~\cite[Proposition~2.7 and Theorem~3.1]{Campana1992}.
We work near a specified compact fiber of a holomorphic
submersion and obtain a metric area bound.

During the construction, the center is allowed to move within $F$.
If compactness has not already produced a sphere meeting $F$ but not
contained in $F$, this freedom allows the normalized base derivative
to pass to the principal map. The final limit used for the explicit area
bound retains the two incidence properties, though not necessarily the
derivative normalization; hence the theorem prescribes neither a
tangent direction nor a contact order.

Choose a holomorphic coordinate map $q$ on a neighborhood of $x$.
Its fiber over $q(x)$ is $F=\{x\}$, so the relative theorem gives the
following pointwise statement.

\begin{theorem}
\label{thm:intro-pointwise-fano}
Let $X$ be a complex Fano manifold of dimension $n\geq1$, equipped
with a K\"ahler metric $\omega$ satisfying
$\Ric(\omega)\geq\kappa\omega$ for some $\kappa>0$. Then, for every $x\in X$, there is a nonconstant holomorphic
map $u_x:\PP^1\to X$ such that
\[
 x\in u_x(\PP^1),\qquad
 \int_{\PP^1}u_x^*\omega\leq\frac{2\pi(n+1)}{\kappa}.
\]
\end{theorem}

Classical bend-and-break produces a rational curve $C$ through a
prescribed point of a Fano $n$-fold with
\[
 0<-K_X\cdot C\leq n+1.
\]
Mori's account~\cite[p.~749]{MoriICM1984} attributes this form of the
result to Koll\'ar, by a technique similar to that of~\cite{Mori1979}.
Integration over $C$ gives
\[
 \kappa\int_C\omega\le\int_C\Ric(\omega)=2\pi(-K_X\cdot C).
\]
The classical result controls the anticanonical degree, giving more
information than the area bound for a fixed metric.
The analytic construction controls area and weighted derivatives during
the process.
For the classical argument, see~\cite[Chapter~3]{Debarre2001}.

\subsection{The disk operations and their quantitative control}

For a holomorphic map $f:\D_R\to X$, write
\[
\Area_R(f)=\int_{\D_R}f^*\omega,\qquad
B_R(f)=\sup_{z\in\D_R}(R-|z|)|f_z(z)|_\omega.
\]
A bound $B_R(f)\le B_0$ gives
$|f_z(z)|_\omega\le B_0/(R-|z|)$, and hence a derivative bound on every
smaller disk. The quantity $B_R(f)$ is central to Brody's original proof
of the reparametrization lemma~\cite[Lemma~2.1]{Brody1978}: a point where
the weighted derivative is nearly maximal supplies the center and scale
for a map on a larger disk with controlled derivatives on compact
subdisks. We use bounded values of $B$ to choose uniform surgery
parameters, and sequences with $B$ tending to infinity to extract
bubbles by point selection.

In the relative construction, we allow the disk center to move within
$F$ while fixing one nonzero scalar component of the derivative of
$q\circ f$ there.

The first surgery lowers area while fixing the center and one scalar
component of the derivative. Its geometric starting point is a paired
second variation. For a disk extending holomorphically across its
boundary, the variation fields are holomorphic sections $\xi$ of
$f^*T_X$, smooth on the closed disk. The K\"ahler condition identifies
the sum of the second area derivatives along the real and imaginary
parameter axes with the boundary quadratic form
\[
 \mathcal H_{R,f}(\xi)\coloneqq
 \int_{\partial\D_R}\partial_\nu|\xi|_\omega^2\,ds
\]
whenever $\xi$ is realized by a holomorphic family
(Lemma~\ref{lem:levi}), where $\nu$ is the outward Euclidean unit
normal and $ds$ is arclength. The paired variation therefore depends
only on the field along the disk.

In the proof of Myers's theorem~\cite{Myers1941}, Ricci curvature
enters by tracing second variations built from a parallel orthonormal
frame along a geodesic~\cite[\S19, Theorem~19.4]{Milnor1963}; for
holomorphic disks the analogous trace calculation uses a holomorphic
frame $e=(e_1,\ldots,e_n)$ of $f^*T_X$ that is smooth on the closed
disk and unitary on its boundary. Such a frame is supplied by
Wiener--Masani spectral factorization~\cite{WienerMasani1957}, in the
form used for Hermitian holomorphic bundles by Berndtsson and
Rosay~\cite[Theorem~2.1 and Proposition~2.3]{BerndtssonRosay2003}:
their curvature-potential estimate, combined with the Green-potential
bound below, controls the normalized frame and its inverse uniformly
under fixed area and weighted-derivative bounds. These bounds feed the
quantitative holomorphic realization of the variation fields; boundary
unitarity is a normalization of the frame and imposes no geometric
hypothesis on $X$.

Let $G$ be the Gram matrix of this frame, with entries
$G_{jk}=\langle e_j,e_k\rangle_\omega$. Since $G=\Id_n$ on the boundary, $\partial_\nu\tr G=\partial_\nu\log\det G$ there.
Green's formula and the determinant-curvature identity therefore give,
before imposing any constraint at the center,
\[
 \sum_{a=1}^n\mathcal H_{R,f}(e_a)
 =\int_{\partial\D_R}\partial_\nu\log\det G\,ds
 =-2\int_{\D_R}f^*\Ric(\omega).
\]
Ricci curvature enters the deformation of a holomorphic disk through
this identity; see \eqref{eq:matrixtrace} with $\Phi=\Id_n$. Imposing the
interpolation at the center adds a Dirichlet cost to it, and the scalar
constraint means keeping one chosen component of
$(q\circ f)_z(0)$ fixed. Write this component as $\lambda(f_z(0))$,
where $\lambda\in T^*_{f(0)}X$ is the corresponding nonzero covector.
A constant unitary change of frame then places
$e_2(0),\ldots,e_n(0)$ in $\ker\lambda$. Vanishing order at least two in the first component and
at least one in the others then suffices to preserve the center and
the specified scalar derivative in the affine families below.
Taking monomials of the lowest admissible degrees $(2,1,\ldots,1)$,
normalized on $|z|=R$, gives Dirichlet cost $\pi(n+1)$ and paired trace
\[
 4\pi(n+1)-2\int_{\D_R}f^*\Ric(\omega)
 \le4\pi(n+1)-2\kappa\Area_R(f).
\]
It follows that, above $\Acrit=2\pi(n+1)/\kappa$, the boundary quadratic form is
negative on at least one of these fields. The realization below turns
this negativity into area decrease. The calculation after
\eqref{eq:scalartrace} proves that this Dirichlet cost is the least
possible under the prescribed interpolation and boundary conditions.

Other choices of field are available: using $(z/R)^m$ in every
direction, for a fixed integer $m\ge2$, gives paired trace
$4\pi mn-2\int_{\D_R}f^*\Ric(\omega)$ and threshold $2\pi mn/\kappa$.
The perturbation then vanishes to order $m$ and preserves the entire
$(m-1)$-jet at the center, while $m=2$ preserves the first jet, at
threshold $4\pi n/\kappa$. These statements concern the disk
deformations only, and not the jets of the final sphere.

To obtain uniform deformations from this calculation, we lift the disk
and its fields to a fixed affine Stein manifold lying over $X$.
Quantitative corona estimates~\cite{Nicolau1990} control the lifts;
Euclidean perturbation followed by a holomorphic neighborhood
retraction \cite{DocquierGrauert1960,Forstneric2010Stratified} and
projection to $X$ gives actual holomorphic families. Radial
regularization makes the fields extend across the source boundary and
preserves the center constraints exactly. For the resulting families,
the sum of the second area derivatives along the real and imaginary
parameter axes differs from the trace calculated above by an error
that can be made arbitrarily small. Fix area and weighted-derivative bounds
$A_0,B_0$ and require $\Area_R(f)\ge\Acrit+\delta$ for fixed
$\delta>0$. This sum remains negative for the realized families,
so at least one real parameter direction
has negative second variation. We choose its sign to make the
first-order contribution nonpositive. A common parameter interval and
cubic remainder now give a positive step $t_*$ and an area decrease
$\sigma>0$, uniform in the disk and its radius. These constants
depend on the fixed bounds and the margin $\delta$; see
Proposition~\ref{cor:compact-fiber-descent}. This uniform decrease
provides room for the area cost of enlargement.

For the second surgery, choose a holomorphic lift $H$ of $f$ into an
affine space and a holomorphic map $\Pi$ back to $X$, defined near the
lifted image, with $\Pi\circ H=f$. For sufficiently small $t>0$, set
\[
 h_t(z)=\Pi(H(z)+tzH'(0)),\qquad
 G_t(z)=h_t(z/(1+t)).
\]
The affine perturbation fixes the center and multiplies $f_z(0)$ by
$1+t$. Rescaling the source then restores the original derivative
and increases the source radius from $R$ to $(1+t)R$.
On these respective source disks, both the area and the squared
weighted derivative change by at most a uniform constant times $t$
under fixed area and weighted-derivative bounds
(Proposition~\ref{prop:finite-area-enlargement}).

Alternating enlargement and area reduction keeps the area below a
fixed cap, with quantitative control of each finite stage in
Proposition~\ref{eff:prop:finite-stage}. If derivatives concentrate
while the source radii remain bounded, compactness either produces
the required sphere or retains a normalized principal map. In the
latter case, a bubble accounts for a definite area loss, and we
enlarge the principal map before continuing. The recovery argument
in Section~\ref{eff:sec:nested-recursion} organizes these returns
and yields the required sphere or disks of arbitrarily large radius;
this passage still uses subsequential limits.

We obtain the stated bound by letting the area cap decrease to $\Acrit$.
During this passage, we mark one point mapping into $F$ and another
mapping outside a fixed neighborhood of $F$. The limit may have several
sphere components; the two markings allow us to select one that meets
$F$ without being contained in it, with area at most $\Acrit$.
The compactness argument is given in Section~\ref{sec:area-descent},
using the theorem of Ivashkovich and
Shevchishin~\cite{IvashkovichShevchishin2000}.

The controlled holomorphic deformations developed here also suggest
an approach to Conjecture~\ref{conj:fano-oka}, in which deformations of
disks would lead to dominating holomorphic sprays on Fano manifolds.

\par\smallskip\noindent{\bfseries Plan of the paper.}\par\nobreak
\smallskip\noindent
Section~\ref{sec:spread-sphere-quotient} proves the relative sphere theorem
and records the value of the bound for the Fubini--Study metric. The Ricci-trace calculation,
uniform holomorphic realization, and marked compactness are established
in Sections~\ref{sec:trace-lifting}--\ref{sec:area-descent}.
The final section recalls how the construction combines with
classical algebraic arguments to recover rational connectedness and
Mori's ample-tangent characterization of projective space.

\section{Explicit construction of a relative sphere}
\label{sec:spread-sphere-quotient}

\subsection{Normalized disks and the two operations}
We fix a complex Fano manifold $X$ of dimension $n\geq1$ and a
K\"ahler metric $\omega$ satisfying $\Ric(\omega)\geq\kappa\omega$
for some $\kappa>0$ throughout the construction.
In holomorphic coordinates $w^1,\ldots,w^n$, set
$h_{\alpha\bar\beta}=\langle\partial/\partial w^\alpha,
\partial/\partial w^\beta\rangle_\omega$ for the pairwise inner products
of the coordinate vector fields; their Gram matrix is
$h=(h_{\alpha\bar\beta})$.
Our Hermitian inner products are complex-linear in the first variable and
conjugate-linear in the second~\cite[Ch.~V, \S12]{DemaillyCADG2012}.
Coefficient columns $a,b$ pair as
$a^{\mathsf T}h\bar b$, where $\mathsf T$ denotes ordinary transpose. For a holomorphic source map $f$, write
$f_z=df(\partial/\partial z)$. With repeated coordinate indices summed and
$dA$ denoting Euclidean area measure, our normalizations are

\begin{equation}
 \omega=\frac{i}{2}h_{\alpha\bar\beta}\,dw^\alpha\wedge d\bar w^\beta,
 \qquad
 f^*\omega=\abs{f_z}_h^2\,dA,
 \qquad
 \Ric(\omega)=-i\partial\bar\partial\log\det h.
 \label{eq:normalization}
\end{equation}
For a holomorphic map $f:S\to X$ from a Riemann surface, set
$\Area(f)\coloneqq\int_S f^*\omega$; this counts geometric area with
the multiplicity of the map. A holomorphic sphere is a map
$u:\PP^1\to X$. When $u$ is nonconstant, we call its reduced image
a rational curve, while $\Area(u)$ continues to include the mapping
multiplicity.
For $a\in\C$ and $r>0$, write
$\D_r(a)\coloneqq\{z\in\C:|z-a|<r\}$, and abbreviate
$\D_R\coloneqq\D_R(0)$ and $\D\coloneqq\D_1(0)$.
We write $\overline\D_R=\{z\in\C:|z|\le R\}$ for the closed disk and
$\partial\D_R=\{z\in\C:|z|=R\}$ for its boundary circle. We use the area on $\D_R$ and the boundary-distance-weighted derivative bound
\begin{equation}
 \Area_R(f)\coloneqq \int_{\D_R}f^*\omega,
 \qquad
 B_R(f)\coloneqq \sup_{z\in\D_R}(R-\abs z)\abs{f_z(z)}_\omega.
 \label{eq:area-brody}
\end{equation}

Write $\cO(V,X)$ for the holomorphic maps from an open set
$V\subset\C$ to $X$, and $K\Subset U$ when the closure of $K$ is
compact and contained in $U$. Smooth convergence means uniform
convergence of all derivatives on compact source subsets.

For this fixed Ricci lower bound, set
\begin{equation}
 \Acrit = \frac{2\pi(n+1)}\kappa.
 \label{eq:ricci-threshold}
\end{equation}

The Fano manifold $X$ is projective. We fix an embedding
$\iota:X\hookrightarrow\PP^{m-1}$ to carry out the two disk operations. It
remains fixed as the map and source radius vary, and the constants
may depend on it. All areas and target derivative norms are still
measured using $\omega$.

For the relative constraint, start with $q:N\to Y$, $y_0$, and
the compact fiber $F=q^{-1}(y_0)$ in
Theorem~\ref{thm:bounded-fiber-escape}. Choose a smooth-boundary
neighborhood $U$ with $F\Subset U\Subset N$, small enough that $q$ is a
submersion on a neighborhood of $\overline U$. Since $F=q^{-1}(y_0)\Subset U$, one has
$y_0\notin q(\partial U)$, and $q(U)$ is an open neighborhood of
$y_0$. We can therefore choose a coordinate ball $Y'$
about $y_0$ such that $\overline{Y'}\cap q(\partial U)=\varnothing$ and
$Y'\subset q(U)$. After replacing $N$ by $U\cap q^{-1}(Y')$ and $Y$ by
$Y'$, the map $q$ is a proper holomorphic submersion: it remains a
submersion by the choice of $U$, and the inverse image of a compact
subset of $Y'$ is closed in $\overline U$ and avoids $\partial U$.
In the restricted coordinate ball write $d_Y=\dim Y$ and let
$\psi:N\to\C^{d_Y}$ be the coordinate expression of $q$, translated so that
$\psi^{-1}(0)=F$. Fix $0\neq\eta\in(\C^{d_Y})^*$ and set
\[
 \lambda^F=(\lambda_x^F)_{x\in F},\qquad
 \lambda_x^F\coloneqq \eta\circ d\psi_x\in T_x^*X.
\]
The covectors $\lambda_x^F$ are the differentials of the same scalar
function $\eta\circ\psi$ at points $x\in F$, so even when $f(0)$
moves in $F$, normalization always means
$(\eta\circ\psi\circ f)_z(0)=1$, with the composition defined near
zero. It follows that $d\psi_{f(0)}(f_z(0))\neq0$, so the initial
tangent vector is not tangent to $F$. The submersion property makes
$\lambda_x^F$ nonzero for every $x\in F$.
The classes below consist of maps on $\D_R$ that extend
holomorphically to some open neighborhood of $\overline\D_R$. The
neighborhood may depend on the map. Set

\begin{equation}\label{eq:moving-root-class}
 \cA_R^{F,\lambda^F}
 =\left\{f\in\cO(\D_R,X):
 f(0)\in F,\quad\lambda^F_{f(0)}(f_z(0))=1\right\}.
\end{equation}
We call a holomorphic disk \emph{normalized} if its center lies in $F$
and its base scalar derivative equals $1$. Membership in
$\cA_R^{F,\lambda^F}$ requires, in addition, holomorphic extension
across the source boundary. In the iteration, $A$ and $B$ always refer
to the current map and its current source radius.

Choose $x_0\in F$ and a holomorphic coordinate chart $\chi$ with
$\chi(x_0)=0$. Since $\lambda^F_{x_0}$ is nonzero, choose $v\in\C^n$
such that $\lambda^F_{x_0}(d(\chi^{-1})_0(v))=1$. For all sufficiently
small $R_0>0$, the map $f_0(z)=\chi^{-1}(zv)$ is holomorphic
near $\overline\D_{R_0}$ and belongs to
$\cA_{R_0}^{F,\lambda^F}$. Moreover,
$\Area_{R_0}(f_0)\to0$ as $R_0\downarrow0$. Fix one such
$R_0$ with $\Area_{R_0}(f_0)<\Acrit$; the pair
$(f_0,R_0)$ is the initial disk for the iteration.

To relate the source radius to the weighted derivative, let
\[
 \ell_F\coloneqq\max_{x\in F}\|\lambda_x^F\|_{\omega,*},
\]
the maximal dual norm of the covector field $\lambda^F$ along $F$.
Here $\|\cdot\|_{\omega,*}$ is the dual norm induced by $\omega$.
Since $F$ is compact and $\lambda_x^F\ne0$ on $F$, we have
$0<\ell_F<\infty$. For every normalized disk, the scalar condition gives
$1=|\lambda_{f(0)}^F(f_z(0))|\le\ell_F|f_z(0)|_\omega$.
Evaluating the supremum defining $B_R(f)$ at $z=0$, we obtain
\begin{equation}\label{eq:radius-progress}
 B_R(f)\ge R|f_z(0)|_\omega\ge R/\ell_F.
\end{equation}
Each surgery keeps the value at $0$ and the normalized scalar derivative
fixed. When a limiting disk is retained for the next stage, its value
at $0$ is still in $F$, but may be a different point of $F$; its scalar
derivative is still $1$. The example in
Section~\ref{subsec:marked-incidence} illustrates this with
$F=\PP^1\times\{[0:1]\}$. The new maps need not agree with
their predecessors on the original source disks: enlargement is
achieved by deformation, not by analytic continuation of a single map.
Fix also $F\Subset V\Subset N$ with smooth boundary.

Fix $c>0$ with $\omega\ge c\iota^*\omega_{\rm FS}$, normalizing a
projective line to have Fubini--Study area $\pi$, and set
\begin{equation}\label{eff:eq:energy-quantum}
 \hbar_X=c\pi>0.
\end{equation}
If $u:\PP^1\to X$ is nonconstant, the projective map $\iota\circ u$
has an integer degree $d\geq1$. Hence
\[
 \Area(u)\geq c\int_{\PP^1}u^*\iota^*\omega_{\rm FS}
 =c\pi d\geq \hbar_X.
\]

The two operations play complementary roles. Surgery I acts on a fixed
source disk and lowers its area above the Ricci threshold, using the
K\"ahler variation identity of Lemma~\ref{lem:levi}. Surgery II enlarges
the source at a linear area cost. The next proposition makes the first
operation uniform and records its effect on $B_R$.

\begin{proposition}\label{cor:compact-fiber-descent}
In the relative setting fixed above, let $A_0,B_0,\delta>0$.
There are constants $\sigma>0$ and $\beta_I<\infty$, depending only on
these three bounds and the fixed relative data, with the following
property. For every $R>0$, each
$f\in\cA_R^{F,\lambda^F}$ with
\[
 \Acrit+\delta\le\Area_R(f)\le A_0,\qquad B_R(f)\le B_0
\]
admits a map $g\in\cA_R^{F,\lambda^F}$ satisfying
\begin{equation}\label{eq:quantitative-descent}
 g(0)=f(0),\qquad
 \Area_R(g)\le\Area_R(f)-\sigma,\qquad
 |B_R(g)^2-B_R(f)^2|\le\beta_I.
\end{equation}
\end{proposition}

The last estimate gives $B_R(g)\le(B_0^2+\beta_I)^{1/2}$, so
Surgery I controls $B_R$ without necessarily decreasing it. The map
$g$ is explicit: it belongs to the affine-lift family of
Theorem~\ref{thm:engine}, and \eqref{eq:descent-step} records its step
size and guaranteed area decrease.

The second surgery perturbs a lifted disk in affine space, projects
back to $X$, and dilates the source.

\begin{proposition}\label{prop:finite-area-enlargement}
Let $(X,\omega)$ be a compact projective manifold with a smooth
Hermitian metric, and fix an embedding
$\iota:X\hookrightarrow\PP^{m-1}$. Fix bounds $A_0,B_0\geq0$.
There are constants $0<\bar t\leq1$ and $C_{II}<\infty$, independent
of the radius and the individual map, with the following properties.
\begin{enumerate}
\renewcommand{\labelenumi}{\textup{(\roman{enumi})}}
\item Let $R>0$ and let $h$ be a map to $X$, holomorphic on a
neighborhood of $\overline\D_R$, with
$\Area_R(h)\leq A_0$ and $B_R(h)\leq B_0$. There are a holomorphic
lift $H:U\to\C^{2m}$ on a neighborhood $U$ of $\overline\D_R$ and a
holomorphic ambient projection $\Pi$ defined near
$H(\overline\D_R)$, with $\Pi\circ H=h$ near $\overline\D_R$, such
that the formulas
\begin{equation}\label{eq:explicit-enlargement}
 h_t(z)=\Pi(H(z)+tzH'(0)),\qquad
 G_t(z)=h_t(z/(1+t)),\qquad 0<t\leq\bar t,
\end{equation}
give holomorphic maps $h_t:\D_R\to X$ and
$G_t:\D_{(1+t)R}\to X$ that extend past their respective closed disks.
The map $G_t$ satisfies
\begin{align}
 |\Area_{(1+t)R}(G_t)-\Area_R(h)|&\leq C_{II}t,
 \label{eff:eq:II-A}\\
 |B_{(1+t)R}(G_t)^2-B_R(h)^2|&\leq C_{II}t/\pi.
 \label{eff:eq:II-B}
\end{align}
The same estimates compare $h_t$ and $h$ on $\D_R$, and
$(h_t)_z(0)=(1+t)h_z(0)$, while $G_t(0)=h(0)$ and
$(G_t)_z(0)=h_z(0)$.

\item Let $h$ be holomorphic only on the finite open disk $\D_R$, with
the same area and derivative bounds. For each $0<t\leq\bar t$, let
$r=(1+t/2)R/(1+t)$ and apply part~\textup{(i)} to $h|_{\D_r}$.
If $H_r$ is the resulting lift, then
\begin{equation}\label{eff:eq:closed-repair}
 R'=(1+t)r=(1+t/2)R,\qquad
 g(z)=\Pi\left(H_r\left(\frac{z}{1+t}\right)
       +t\frac{z}{1+t}H_r'(0)\right)\quad (z\in\D_{R'})
\end{equation}
extends holomorphically past $\overline\D_{R'}$. Since $r<R<R'$, its
source is larger than that of $h$, and
\begin{equation}\label{eff:eq:closed-repair-cost}
 \Area_{R'}(g)\leq\Area_R(h)+C_{II}t,\qquad
 B_{R'}(g)^2\leq B_R(h)^2+C_{II}t/\pi.
\end{equation}
Moreover, $g(0)=h(0)$ and $g_z(0)=h_z(0)$.
\end{enumerate}
\end{proposition}

Propositions~\ref{cor:compact-fiber-descent}
and~\ref{prop:finite-area-enlargement} are proved in
Section~\ref{subsec:disk-operations}.

In \eqref{eq:explicit-enlargement}, the affine perturbation changes the
derivative on the fixed source disk, whereas the subsequent dilation enlarges
the source and restores the original derivative without changing the perturbed
image or its area.

\subsection{Principal limits and marked values}

We first construct a principal map in the original source coordinate,
then a rational chain retaining two marked values. If the principal map
is entire and has finite area, removal at infinity extends it to a
sphere. We derive these statements from the compactness and removal
results of Ivashkovich and
Shevchishin~\cite{IvashkovichShevchishin2000}.
Proposition~\ref{prop:principal-compactness},
Theorem~\ref{thm:removal}, and
Lemma~\ref{lem:two-marked-expanding-chain} are proved in
Section~\ref{sec:area-descent}.

\begin{proposition}
\label{prop:principal-compactness}
Let $(X,\omega)$ be a compact Hermitian manifold. For each $j$, let
$R_j>0$ and let
$f_j:U_j\to X$ be holomorphic, where $U_j\subset\C$ is open and
$\overline\D_{R_j}\Subset U_j$. Assume that $\sup_j\Area_{R_j}(f_j)<\infty$,
and set $A_\infty=\liminf_j\Area_{R_j}(f_j)$. Choose $R_\infty\in(0,\infty]$
with $R_\infty\leq\liminf_jR_j$, interpreting $\D_\infty$ as $\C$.
After passing to a subsequence realizing $A_\infty$ and then to a further
subsequence, there are a finite set
$\Sigma_\infty\subset\D_{R_\infty}$ and a holomorphic map
$f_\infty:\D_{R_\infty}\to X$ such that, in the original source coordinate,
\[
 f_j\longrightarrow f_\infty
 \quad\text{smoothly on compact subsets of }
 \D_{R_\infty}\setminus\Sigma_\infty,
 \qquad \Area_{R_\infty}(f_\infty)\leq A_\infty.
\]
The possibly constant map $f_\infty$ is holomorphic across
$\Sigma_\infty$.
\end{proposition}

The extension across $\Sigma_\infty$ does not assert convergence of
$f_j$ at the concentration points.

\begin{theorem}\label{thm:removal}
Let $(X,\omega)$ be a compact Hermitian manifold. Every holomorphic map
from a punctured disk to $X$ with finite $\omega$-area extends across
the puncture. Every finite-area holomorphic map $u:\C\to X$ extends to
a holomorphic map $\PP^1\to X$.
\end{theorem}

\begin{lemma}
\label{lem:two-marked-expanding-chain}
Let $(X,\omega)$ be a compact Hermitian manifold, and let $K_0,K_1\subset X$
be disjoint compact subsets.
Suppose $R_j>1$, $R_j\to\infty$, and the holomorphic maps $f_j:U_j\to X$, where
$U_j\subset\C$ is open
and $\overline\D_{R_j}\Subset U_j$, satisfy
\[
 f_j(0)\in K_0,\qquad f_j(1)\in K_1,
 \qquad \sup_j\Area_{R_j}(f_j)<\infty.
\]
Then there are nonconstant holomorphic maps $u_1,\ldots,u_\ell:\PP^1\to X$
whose images form a chain: the first meets $K_0$, the last meets $K_1$, and
consecutive images intersect. Their total area, counted with the
multiplicities of these maps, satisfies
\[
 \sum_{i=1}^{\ell}\Area(u_i)
 \leq\liminf_j\Area_{R_j}(f_j).
\]
\end{lemma}

The disjoint target sets force the path between the two marked values
to contain a nonconstant component. Equality of nodal values preserves
incidence when constant components along that path are omitted.

The scalar normalization is imposed on the local base function:
$(\eta\circ\psi\circ f)_z(0)=1$.
If first-exit points from a fixed neighborhood of $F$ approach zero,
rescaling and marked compactness give a sphere meeting $F$ but not
contained in $F$.
Otherwise, a subsequence remains in that neighborhood on a fixed source
disk, where the functions $\eta\circ\psi\circ f_j$ are uniformly bounded.
Cauchy estimates then pass their derivatives at zero to the principal
limit.

\begin{lemma}
\label{lem:relative-compactness-alternative}
Let $(X,\omega)$ be a compact Hermitian manifold, and assume the relative
setting fixed above: $F\subset N\subset X$ is the nonempty compact fiber of
the proper holomorphic submersion $q:N\to Y$, $\psi$ is its local coordinate
expression with $F=\psi^{-1}(0)$, and
$\lambda_x^F=\eta\circ d\psi_x$ for the fixed $\eta\ne0$.
Suppose
$R_j\to R_\infty\in(0,\infty]$, where $\D_\infty=\C$, and
$f_j\in\cA_{R_j}^{F,\lambda^F}$
(see \eqref{eq:moving-root-class}), with
\[
 \Area_{R_j}(f_j)\longrightarrow E<\infty.
\]
Then at least one of the following conclusions holds:
\begin{enumerate}
\renewcommand{\labelenumi}{\textup{(\roman{enumi})}}
\item There is a nonconstant sphere $v:\PP^1\to X$ such that
\[
 v(\PP^1)\cap F\neq\varnothing,\qquad
 v(\PP^1)\not\subset F,\qquad \Area(v)\leq E.
\]
\item A subsequence has a principal limit
$h:\D_{R_\infty}\to X$ in the original source coordinate, as in
Proposition~\ref{prop:principal-compactness}, satisfying
\[
 \Area_{R_\infty}(h)\leq E,\qquad h(0)\in F,\qquad
 \lambda_{h(0)}^F(h_z(0))=1.
\]
On some fixed disk about zero, the base compositions
$\psi\circ f_j$ converge to $\psi\circ h$ smoothly on compact subsets.
\end{enumerate}
\end{lemma}

\begin{proof}
Choose $F\Subset V\Subset N$ with smooth boundary, and fix $r>0$ such
that $R_j>r$ for all sufficiently large $j$.
There are two possibilities after passage to a subsequence.

First, suppose there are $z_j\to0$ with $z_j\neq0$ and
$f_j(z_j)\in\partial V$. Such points can be chosen as first exits along
segments from zero whenever no fixed small closed source disk has image
in $V$ along an infinite subsequence. The rescaled maps
\[
 v_j(\zeta)\coloneqq f_j(z_j\zeta),
 \qquad |\zeta|<R_j/|z_j|,
\]
extend past their closed source disks and satisfy
\[
 v_j(0)\in F,\qquad v_j(1)\in\partial V,\qquad
 \Area_{R_j/|z_j|}(v_j)=\Area_{R_j}(f_j)\longrightarrow E.
\]
Their source radii tend to infinity. Lemma~\ref{lem:two-marked-expanding-chain}
gives a chain of nonconstant spheres from $F$ to $\partial V$, of total
area at most $E$. Take its first component not contained in $F$. It meets
$F$, either at the first marked value or at its intersection with the
preceding component, and has area at most $E$. This proves \textup{(i)}.

Otherwise, for some $0<\rho<r$, an infinite subsequence satisfies
$f_j(\overline\D_\rho)\subset V$. Apply
Proposition~\ref{prop:principal-compactness} on $\D_{R_\infty}$ to this
subsequence. Its principal map $h$ takes values in $\overline V$ on
$\D_\rho$: this holds away from the finite concentration set by
convergence, and at that set by continuity. In particular, $\psi\circ h$
is defined there. Since $\overline V\Subset N$, the holomorphic maps
$\psi\circ f_j$ are uniformly bounded on $\D_\rho$. A further subsequence
converges locally uniformly to a holomorphic map, which equals
$\psi\circ h$ away from the finite concentration set by principal
convergence and hence throughout $\D_\rho$ by the identity theorem.
Cauchy estimates give convergence of all derivatives of these base maps
on compact subdisks, even if $0$ is a concentration point. This does not
exclude a bubble at $0$ inside a fiber of $\psi$. Since $f_j(0)\in F$ and
$(\eta\circ\psi\circ f_j)_z(0)=1$, we obtain
\[
 \psi(h(0))=0,\qquad
 (\eta\circ\psi\circ h)_z(0)
 =\lim_j(\eta\circ\psi\circ f_j)_z(0)=1.
\]
The area bound follows from principal compactness, proving \textup{(ii)}.
\end{proof}

The following standard consequence of local small-energy regularity supplies
the weighted derivative bound for a principal map on a finite disk; we include
the short argument, since the resulting bound may depend on the individual
map.

\begin{lemma}
\label{lem:finite-area-weighted-bound}
Let $(X,\omega)$ be a compact Hermitian manifold, let $0<R<\infty$, and
let $h:\D_R\to X$ be holomorphic with $\Area_R(h)<\infty$. Then
$B_R(h)<\infty$.
\end{lemma}

The proof in Section~\ref{subsec:weighted-derivative-proof} applies the
small-energy estimate of Ivashkovich and
Shevchishin~\cite[Lemma~1.1]{IvashkovichShevchishin2000} in a
boundary collar.

Consequently, whenever compactness retains a principal map $h$ on a
finite disk, one may set $B_0=B_R(h)$ and apply
Proposition~\ref{prop:finite-area-enlargement} with this derivative bound
and the current area bound. The enlargement parameters are chosen only
after $h$ is known and may therefore vary from one principal map to the
next. The two-sided $B^2$ estimate \eqref{eff:eq:II-B} applies to a
closed-disk enlargement; when $h$ is defined only on its open disk, we
use the repair and the upper bound in
\eqref{eff:eq:closed-repair-cost}.

\subsection{Recovery with an area loss}

\begin{lemma}\label{eff:lem:loss-recovery}
In the relative setting fixed above, let $R_j>0$ and let
$f_j:\D_{R_j}\to X$ extend holomorphically past
$\overline\D_{R_j}$, with $f_j(0)\in F$ and
$\lambda^F_{f_j(0)}((f_j)_z(0))=1$. Assume $R_j\to R\in(0,\infty)$,
$\Area_{R_j}(f_j)\to E<\infty$, and $B_{R_j}(f_j)\to\infty$.
One may select either a nonconstant sphere $v:\PP^1\to X$ satisfying
\[
 v(\PP^1)\cap F\neq\varnothing,\qquad
 v(\PP^1)\not\subset F,\qquad \Area(v)\leq E,
\]
or a normalized principal map $h:\D_R\to X$ satisfying
\begin{equation}\label{eff:eq:recovery-loss}
 \Area_R(h)\le E-\hbar_X,
 \qquad B_R(h)<\infty.
\end{equation}
The same conclusion holds if the $f_j$ are holomorphic only on the
open disks $\D_{R_j}$, without extension across their boundaries;
all other hypotheses are unchanged.
\end{lemma}
\begin{proof}
Apply Lemma~\ref{lem:relative-compactness-alternative}. If its first
alternative gives the required sphere, we are done.
Otherwise, pass to the subsequence in its second alternative. It has a
normalized principal map $h:\D_R\to X$ with $\Area_R(h)\le E$ and a
finite concentration set $\Sigma$. We use the same subsequence below,
so its areas still tend to $E$ and its weighted derivatives still
diverge.

Apply Lemma~\ref{lem:pointselect} to this subsequence. In its
construction, $w_j\in\D_{R_j}$ maximizes the weighted derivative. Put
$d_j=R_j-|w_j|$ and $\Lambda_j=|(f_j)_z(w_j)|_\omega$. Then
\[
 d_j\Lambda_j=B_{R_j}(f_j)\longrightarrow\infty,
 \qquad \Lambda_j\ge B_{R_j}(f_j)/R_j\longrightarrow\infty.
\]
The rescaled maps
$u_j(\zeta)=f_j(w_j+\zeta/\Lambda_j)$ converge locally smoothly on
$\C$ to a nonconstant holomorphic map $u$. Pass to a further
subsequence with $w_j\to w\in\overline\D_R$. If $w$ were in
$\D_R\setminus\Sigma$, the smooth convergence to $h$ near $w$ would
bound $\Lambda_j$. Thus $w\in\Sigma\cup\partial\D_R$.

We now count the areas of $h$ and $u$ on disjoint parts of the same
source disks. Fix $S>0$ and a compact region
$Q\subset\D_R\setminus\Sigma$. For large $j$, the disk
$\Delta_{j,S}=\D_{S/\Lambda_j}(w_j)$ lies in $\D_{R_j}$ because
$d_j\Lambda_j\to\infty$. It shrinks to $w$, so it is disjoint from
$Q$. The compact set $Q$ is also contained in $\D_{R_j}$ for large $j$. Hence
\[
 \int_Q f_j^*\omega+\int_{\Delta_{j,S}}f_j^*\omega
 \le\Area_{R_j}(f_j).
\]
The first integral converges to $\int_Q h^*\omega$ by smooth principal
convergence. The source change $z=w_j+\zeta/\Lambda_j$ gives
$\int_{\Delta_{j,S}}f_j^*\omega=\Area_S(u_j)\to\Area_S(u)$. Thus
\[
 \int_Q h^*\omega+\Area_S(u)\le E.
\]
Exhaust $\D_R\setminus\Sigma$ by such regions $Q$, then let
$S\to\infty$. The finite set $\Sigma$ contributes no area to the
holomorphic map $h$, so
\[
 \Area_R(h)+\Area(u)\le E.
\]
This also covers $w\in\partial\D_R$: every fixed $Q$ remains a
positive distance from $w$. The map $u$ has finite area and thus
extends to a nonconstant sphere by Theorem~\ref{thm:removal}. The
projective degree estimate following \eqref{eff:eq:energy-quantum}
gives $\Area(u)\ge \hbar_X$. Therefore
$\Area_R(h)\le E-\hbar_X$. Finally,
Lemma~\ref{lem:finite-area-weighted-bound} gives $B_R(h)<\infty$.

This proves the closed-disk assertion. For the open-disk assertion in
the last sentence of the lemma, the maps need not extend across
$\partial\D_{R_j}$. Their areas converge to $E<\infty$, so
Lemma~\ref{lem:finite-area-weighted-bound} gives
$B_{R_j}(f_j)<\infty$ for all large $j$. Choose $w_j\in\D_{R_j}$
with
$(R_j-|w_j|)|(f_j)_z(w_j)|_\omega\ge B_{R_j}(f_j)/2$, and put
$d_j=R_j-|w_j|$. Since the area of each disk is finite, we can choose
$r_j<R_j$ with
\[
 R_j-r_j<\min\{1/j,d_j/2\},\qquad
 \Area_{R_j}(f_j)-\Area_{r_j}(f_j)<1/j.
\]
Then $r_j>|w_j|$, and the restriction to $\D_{r_j}$ extends past its
closed disk. Moreover,
\[
 B_{r_j}(f_j)\ge(r_j-|w_j|)|(f_j)_z(w_j)|_\omega
 \ge\tfrac14 B_{R_j}(f_j)\longrightarrow\infty.
\]
The restrictions retain the center and scalar normalization; their
radii tend to $R$ and their areas tend to $E$. Thus the closed-disk
case just proved applies.
\end{proof}

In the principal-map alternative, recovery keeps the whole principal
map, discards the rescaled sphere, and allows the center to move in
$F$; it is a subsequential operation on a sequence of disks, not a
cutting-and-gluing operation on one given disk, and it therefore
carries no closeness assertion near a bubble.

\subsection{Iteration and the relative sphere}\label{subsec:finite-stages}

Fix $\delta>0$ and set
\begin{equation}\label{eff:eq:window}
 L=\Acrit+\delta,\qquad \eta_0=\delta/2,
 \qquad C=\Acrit+2\delta.
\end{equation}
For a stage threshold $M>0$, choose all constants with the same area
and derivative caps $C,M$. Use the descent step $\sigma$ from
\eqref{eq:descent-step} and set
\begin{equation}\label{eff:eq:tau}
 \tau_M=\min\left\{\frac{\bar t}{2},
             \frac{\eta_0}{\max\{1,C_{II}\}}\right\}>0.
\end{equation}
One enlargement multiplies the radius by $1+\tau_M$ and adds at most
$\eta_0$ to the area. Since $L+\eta_0=\Acrit+3\delta/2<C$, enlarging
a disk of area below $L$ keeps it below the cap $C$. Apply the first
applicable rule, restarting with the updated disk after each surgery:
\begin{enumerate}
\item If $B_R(f)\geq M$, record the disk and end the stage.
\item If $B_R(f)<M$ and $\Area_R(f)\geq L$, perform Surgery I.
\item If $B_R(f)<M$ and $\Area_R(f)<L$, perform Surgery II with
$t=\tau_M$.
\end{enumerate}

\begin{proposition}\label{eff:prop:finite-stage}
In the relative setting fixed above, let $\delta>0$ and
$M>0$. Define $L,\eta_0,C$ by \eqref{eff:eq:window}, choose the surgery
constants for the caps $(C,M)$, and set $\sigma,\tau_M$ as in
\eqref{eq:descent-step} and \eqref{eff:eq:tau}. Let the starting disk
$f_s\in\cA_{R_s}^{F,\lambda^F}$ have radius $R_s>0$ and area
$A_s=\Area_{R_s}(f_s)<C$. The preceding rules
produce a disk with $A<C$ and $B\geq M$ in finitely many surgeries.
If $B_{R_s}(f_s)<M$, then \eqref{eq:radius-progress} gives
\[
 R_s/\ell_F\leq B_{R_s}(f_s)<M.
\]
Hence $R_s<\ell_FM$, so the numerator below is positive and
$N_M\geq1$. Set
\begin{align}
 N_M&=\left\lceil\frac{\log(\ell_F M/R_s)}
                         {\log(1+\tau_M)}\right\rceil,\label{eff:eq:N}\\
 J_0&=\begin{cases}
 0,&A_s<L,\\
 \left\lfloor(A_s-L)/\sigma\right\rfloor+1,&A_s\geq L,
 \end{cases}\label{eff:eq:J0}\\
 J&=\left\lfloor\eta_0/\sigma\right\rfloor+1.
 \label{eff:eq:J}
\end{align}
An explicit upper bound for the total number of surgeries is
\begin{equation}\label{eff:eq:total-count}
 T_M=J_0+N_M(1+J).
\end{equation}
If the starting disk already has $B\geq M$, no surgery is required.
\end{proposition}
\begin{proof}
Both surgeries preserve normalization. Descent lowers area, and
enlargement starts below $L$ and ends below $L+\eta_0<C$. Thus $A<C$
throughout. Each operation starts with $B<M$, so the chosen constants
apply until the stage ends.

At most $J_0$ initial descents bring $A<L$ or end the stage. After each
enlargement, at most $J$ descents do the same; the extra $1$ in the
counts covers equality at $L$. Allow at most $N_M$ such blocks. At the
$N_M$th enlargement, if needed,
\[
 R_s(1+\tau_M)^{N_M}\geq\ell_F M,
\]
so \eqref{eq:radius-progress} gives $B\geq M$. This proves the bound
$T_M$. Summing the area changes over any completed prefix, with
$N_I,N_{II}$ denoting its surgery counts, also gives
\begin{equation}\label{eff:eq:area-ledger}
 0\leq A
       \leq A_s-N_I\sigma+N_{II}\eta_0,
 \qquad N_I\sigma\leq A_s+N_{II}\eta_0.
\end{equation}
\end{proof}

\par\medskip\noindent{\bfseries A finite hierarchy of principal recoveries.}
\label{eff:sec:nested-recursion}\par\nobreak
\smallskip\noindent
A finite stage can make $B$ arbitrarily large without forcing the
source radius to grow beyond a prescribed value. At a bounded limiting
radius, Lemma~\ref{eff:lem:loss-recovery} either gives the required sphere
or returns a principal map whose area has dropped by at least $\hbar_X$.
The principal map may be defined only on an open disk. Repair makes it
eligible for another finite stage but can increase its area; we therefore
record each principal map as the area record \emph{before} repairing it.
The argument bounds only how deeply the recoveries nest; it does not
bound the total number of finite stages or subsequence selections.

\begin{proposition}\label{eff:prop:hierarchy}
In the relative setting fixed above, let $\delta>0$ and
$C=\Acrit+2\delta$ as in \eqref{eff:eq:window}, and let $\hbar_X$ be the
area quantum in \eqref{eff:eq:energy-quantum}. For every finite
$T>0$ and every $f\in\cA_R^{F,\lambda^F}$ with $R>0$ and
$\Area_R(f)<C$, the finite stages and principal recoveries give either a nonconstant
sphere $v:\PP^1\to X$ satisfying
$v(\PP^1)\cap F\neq\varnothing$ and
$v(\PP^1)\not\subset F$, or a normalized disk of radius at least
$T$. The resulting sphere or disk has area at most $C$, and source
radii never decrease. At most $\lfloor C/\hbar_X\rfloor+1$ recovery levels
are nested. The resulting disk may be defined only on its open source
disk.
\end{proposition}

\begin{proof}
We first describe one recovery level. If the starting radius is at
least $T$, there is nothing to do. Otherwise apply
Proposition~\ref{eff:prop:finite-stage} successively with derivative
thresholds $M=1,2,\ldots$, stopping if a radius reaches $T$. If this
never happens, the resulting closed-disk maps have areas below $C$,
nondecreasing radii tending to some $R_*\leq T$, and weighted derivatives
tending to infinity. Pass to a subsequence with convergent areas.
Lemma~\ref{eff:lem:loss-recovery} gives either the required sphere or
a normalized principal map $h:\D_{R_*}\to X$ with
$\Area_{R_*}(h)\leq C-\hbar_X$. If $R_*=T$, this disk is a required output.
One level therefore either completes the proposition or returns a principal
disk of radius less than $T$ and area at most $C-\hbar_X$.

We now iterate this alternative. More precisely, let $P_k$ be the
assertion that, starting from any normalized closed-disk map of area
less than $C$, at most $k$ nested recovery levels give either a
required sphere or disk, or a normalized principal disk $h:\D_{R_h}\to X$
with
\[
 R_h<T,\qquad \Area_{R_h}(h)\leq C-k\hbar_X,\qquad B_{R_h}(h)<\infty.
\]
Source radii do not decrease. The preceding paragraph proves $P_1$;
the finiteness of $B_{R_h}(h)$ follows from
Lemma~\ref{lem:finite-area-weighted-bound}.

Assume $P_k$. Apply it to the starting disk. There is nothing further
to prove if it gives a required sphere or disk. Otherwise denote its
principal output by $(h_1,R_1)$. We may repeat the following operation
as long as the output radius remains below $T$. The principal outputs
have finite $B$, but no common bound for it is known. We therefore
prepare the repair constants for every possible integer derivative cap:
for each $q\geq1$, take $\bar t_q$ and $C_{II,q}$ from
Proposition~\ref{prop:finite-area-enlargement} with area cap $C$ and
$B$-cap $q$, and set
\[
 t_q=\min\left\{\bar t_q/2,\,
       \frac{\hbar_X}{2\max\{1,C_{II,q}\}}\right\}>0.
\]
Given a principal output $(h_j,R_j)$, put $b_j=B_{R_j}(h_j)$ and
$q_j=\max\{1,\lceil b_j\rceil\}$. Then $B_{R_j}(h_j)\leq q_j$
and $\Area_{R_j}(h_j)\leq C$, so the repair estimates chosen for
the caps $(C,q_j)$ apply to $h_j$.
Apply the open-disk repair \eqref{eff:eq:closed-repair} to $h_j$ with
$t=t_{q_j}$. It produces a normalized map $g_j$ holomorphic past its
closed source disk, with radius $(1+t_{q_j}/2)R_j$ and
\[
 \Area(g_j)\leq C-k\hbar_X+\hbar_X/2<C.
\]
If this radius reaches $T$, retain $g_j$. Otherwise apply $P_k$ to
$g_j$. A required sphere or disk ends the construction; in the
remaining case record the next principal output $(h_{j+1},R_{j+1})$,
again \emph{before} repairing it. Its area is at most $C-k\hbar_X$, and
\[
 R_{j+1}\geq(1+t_{q_j}/2)R_j.
\]

Suppose this repetition continues indefinitely without reaching $T$.
Then $R_j$ increases to some $R_*\leq T$. We claim that $b_j\to\infty$.
For each integer $Q\geq1$, set
$a_Q=\min_{1\leq q\leq Q}t_q/2>0$. Every index with $b_j\leq Q$
contributes a radius factor of at least $1+a_Q$. Since all radii stay
below $T$,
\[
 \#\{j:b_j\leq Q\}
 \leq\frac{\log(T/R_1)}{\log(1+a_Q)}<\infty.
\]
This proves the claim. Choose a subsequence of the recorded maps whose
areas converge to $E\leq C-k\hbar_X$. The open-disk assertion of
Lemma~\ref{eff:lem:loss-recovery} now gives either a required sphere or
a normalized principal map on $\D_{R_*}$ of area at most
$E-\hbar_X\leq C-(k+1)\hbar_X$. If $R_*=T$, retain that disk; otherwise this
is the remaining alternative in $P_{k+1}$. The repeated calls to $P_k$
do not increase the nesting depth: only this last principal recovery
adds one level.

Hence $P_k$ holds for every $k\geq1$. With
$k=\lfloor C/\hbar_X\rfloor+1$, its remaining principal-disk alternative
would have negative area and is impossible. A required sphere or disk
must therefore occur. The construction may use countably many finite
stages and subsequence selections, while its recovery depth has the
stated finite bound.
\end{proof}

\begin{proof}[Proof of Theorem~\ref{thm:bounded-fiber-escape}]
Choose $F\Subset V\Subset N$ as above and put
$\rho=1+\sup_{\overline V}|\eta\circ\psi|$.
Every normalized disk $f$ of radius greater than $\rho$ must leave $V$
before source radius $\rho$. Indeed, suppose that
$f(\D_\rho)\subset V$ and set $b=\eta\circ\psi\circ f$. Then $b$ is
holomorphic on $\D_\rho$, while normalization gives
$b(0)=0$ and $b'(0)=1$. Moreover,
\[
 \sup_{\D_\rho}|b|
 \leq\sup_{\overline V}|\eta\circ\psi|=\rho-1.
\]
For every $0<r<\rho$, Cauchy's estimate for the derivative at the
origin therefore gives
\[
 1=|b'(0)|
 \leq\frac{1}{r}\sup_{|z|=r}|b(z)|
 \leq\frac{\rho-1}{r}.
\]
Letting $r\uparrow\rho$ yields
$1\leq(\rho-1)/\rho<1$, a contradiction. Hence there is
$\zeta\in\D_\rho$ with $f(\zeta)\notin V$. Since $f(0)\in F\subset V$,
continuity gives a first exit along the segment $t\zeta$, $0\leq t\leq1$:
if
\[
 t_0=\sup\{t\in[0,1]:f(s\zeta)\in V\text{ for every }0\leq s<t\},
\]
then $z=t_0\zeta$ satisfies
$0<|z|<\rho$ and $f(z)\in\partial V$.

For each $j\geq2$, set $\delta_j=1/(2j)$ and
$C_j=\Acrit+1/j$. Apply
Proposition~\ref{eff:prop:hierarchy} with cap $C_j$ and target radius $T_j=2j\rho$, starting from the same
seed, whose area is less than $\Acrit$.
If it returns a disk $h_j$ of radius at least $T_j$, choose its first
exit $z_j$ as above and set $g_j(\zeta)=h_j(z_j\zeta)$.
This map is defined on a disk of radius greater than $2j$, so its
restriction to $\D_j$ extends past $\overline\D_j$ even when
$h_j$ is defined only on its open disk. If instead the hierarchy returns a sphere $u$, then
$u(\PP^1)$ meets $\partial V$. Indeed, if $u(\PP^1)\subset V$, then
$\psi\circ u:\PP^1\to\C^{d_Y}$ is holomorphic and hence constant.
Since $u(\PP^1)$ meets $F=\psi^{-1}(0)$, this constant is zero, so
$u(\PP^1)\subset F$, contrary to the conclusion of the hierarchy.
Choose an affine coordinate on the sphere in which selected preimages
of $F$ and $\partial V$ are $0$ and $1$, and restrict to $\D_j$ to
define $g_j$.
In both cases,
\[
 g_j(0)\in F,\qquad g_j(1)\in\partial V,\qquad
 \Area_j(g_j)\leq\Acrit+1/j.
\]
Lemma~\ref{lem:two-marked-expanding-chain} gives a chain between these
target sets of total area at most $\Acrit$. Its first component not
contained in $F$ meets $F$ at the first marked value or at the preceding
node, and is the required sphere.
\end{proof}

\begin{proof}[Proof of Theorem~\ref{thm:intro-pointwise-fano}]
Given $x\in X$, let $N$ be a coordinate neighborhood of $x$ and let
$q:N\to Y\subset\C^n$ be its coordinate map. The fiber over $q(x)$ is
$F=\{x\}$, and $q$ is a submersion. Theorem~\ref{thm:bounded-fiber-escape}
therefore gives a nonconstant sphere through $x$ with area at most
$\Acrit$.
\end{proof}

\begin{remark}
\label{rem:metric-bound-example}
On $\PP^n$, use the Fubini--Study form
\[
 \omega_{\mathrm{FS}}
 =\frac{i}{2}\partial\bar\partial
   \log(1+|w^1|^2+\cdots+|w^n|^2).
\]
In the normalization \eqref{eq:normalization},
$\Ric(\omega_{\mathrm{FS}})=2(n+1)\omega_{\mathrm{FS}}$ and
$\int_{\PP^1}\omega_{\mathrm{FS}}=\pi$ on a projective line.
Thus, with $\kappa=2(n+1)$, the bound in
Theorem~\ref{thm:bounded-fiber-escape} becomes $\pi$.
Every nonconstant sphere in $\PP^n$ has positive integer degree and
hence area at least $\pi$, so this value cannot be decreased.
\end{remark}

\section{Ricci traces and scalar-constrained variations}
\label{sec:trace-lifting}

We now use positive Ricci curvature to force area decrease above
$\Acrit$. The calculation combines a boundary trace identity with a
choice of holomorphic multipliers that fix the center and one scalar
derivative. Their Dirichlet cost $\pi(n+1)$,
together with Proposition~\ref{prop:matrixtrace}, gives the threshold
$\Acrit=2\pi(n+1)/\kappa$. Section~\ref{sec:analytic-engine} realizes
the resulting variation fields by holomorphic families with uniform
estimates. We retain the normalizations
\eqref{eq:normalization}--\eqref{eq:area-brody}.

For $z=x+iy$, the complex derivative and the Euclidean Laplacian are
\[
 \partial_z=\tfrac12(\partial_x-i\partial_y),\qquad
 \Delta=\partial_x^2+\partial_y^2=4\partial_z\partial_{\bar z}.
\]
On $f^*T_X$, $\nabla_z$ denotes the Chern covariant derivative induced
by $\omega$ in the source direction.
On $\partial\D_R$, $\nu$ is the outward Euclidean unit normal,
$\partial_\nu$ the corresponding derivative, and $ds$ Euclidean
arclength.
Let $A^*$ denote the conjugate transpose of $A$, and let $\Id_n$
denote the identity matrix. We use the Hilbert--Schmidt norm and the
operator norm
\begin{equation}
\|A\|_{\mathrm{HS}}^2\coloneqq\tr(AA^*)=\sum_{j,k}|A_{jk}|^2,\qquad
\|A\|_{\mathrm{op}}\coloneqq\sup_{|v|=1}|Av|.
\label{eq:matrix-norms}
\end{equation}
Distances and diameters are measured with the K\"ahler metric $\omega$.

For a Hermitian holomorphic line bundle $(L,h)$, let $\Theta(L)$ denote
the curvature of its Chern connection. We use the real curvature form
$i\Theta(L)=-i\partial\bar\partial\log|s|_h^2$, where $s$ is a
nonvanishing local holomorphic frame, and its cohomology class is
$2\pi c_1(L)$~\cite[Ch.~V, (12.8), (13.3)]{DemaillyCADG2012}.
We write $c_1(X)=c_1(K_X^{-1})$.

The following proposition is due to Yau~\cite{Yau1978}; it is an
immediate consequence of his prescribed-Ricci theorem. For the reader's
convenience, we include a proof.

\begin{proposition}[Yau]\label{prop:fano-metric}
Every complex Fano manifold admits a K\"ahler metric $\omega$ and a
constant $\kappa>0$ such that
\begin{equation}
 \Ric(\omega)\geq\kappa\omega.
 \label{eq:fano-metric}
\end{equation}
\end{proposition}

\begin{proof}
Ampleness of $K_X^{-1}$ gives a positive closed $(1,1)$-form $\rho$
representing $2\pi c_1(X)$. In a fixed K\"ahler class, Yau's
prescribed-Ricci theorem gives a metric $\omega$ with
$\Ric(\omega)=\rho$. The least eigenvalue of $\rho$ relative to $\omega$
is positive and continuous on the compact manifold. Its positive minimum
therefore supplies $\kappa$ with $\Ric(\omega)\geq\kappa\omega$.
\end{proof}

The size of $\kappa$ depends on the scale:
$\Ric(c\omega)=\Ric(\omega)$ for $c>0$, so rescaling replaces
$\kappa$ by $\kappa/c$. In the class $[\omega]=2\pi c_1(X)$,
Tian's anticanonical $\alpha$-invariant gives a quantitative choice:
for every $0<\kappa<\min\{1,(n+1)\alpha(X)/n\}$, one can choose
$\omega$ satisfying \eqref{eq:fano-metric}; see~\cite{Tian1987} and~\cite[p.~324]{Szekelyhidi2011}. This concerns the choice of metric;
our construction works with any fixed metric satisfying the stated
Ricci bound.

\begin{remark}\label{rem:fano-converse}
On a Fano manifold, \eqref{eq:fano-metric} specifies the metric used
in the construction. The form $\Ric(\omega)$ is the Chern curvature of
the induced metric on $K_X^{-1}$. Its positivity expresses the ampleness
of $K_X^{-1}$ in differential-geometric terms, in accordance with
Kodaira's embedding theorem~\cite[pp.~314--316]{KodairaPNAS1954}.
\end{remark}

In Euclidean space, the paired second area variation is Green's formula
for the squared norm of the variation field.
The \emph{K\"ahler condition} $d\omega=0$ makes the same boundary
identity exact for a curved target, by the classical coefficient
symmetries of closed $(1,1)$-forms~\cite[Ch.~VI, \S4]{DemaillyCADG2012}.

\begin{lemma}\label{lem:levi}
Let $X$ be a complex manifold equipped with a K\"ahler metric $\omega$,
and fix $R>0$.
Let $\Delta_t\subset\C$
be a parameter disk centered at zero,
$U\subset\C$ an open neighborhood of $\overline{\D}_R$, and
$\mathcal F:U\times\Delta_t\to X$ a holomorphic map. Write
$f_t(z)\coloneqq \mathcal F(z,t)$ and
$\xi\coloneqq \partial_t\mathcal F(\cdot,0)$.  Denote the second derivatives along the two real parameter axes by
\[
 H_{\R}\coloneqq \left.\frac{d^2}{ds^2}\right|_{s=0}\Area_R(f_s),
 \qquad
 H_{i\R}\coloneqq \left.\frac{d^2}{ds^2}\right|_{s=0}\Area_R(f_{is}).
\]
Then
\begin{equation}
 H_{\R}+H_{i\R}
 =\int_{\partial\D_R}\partial_\nu\abs{\xi}_\omega^2\,ds.
 \label{eq:levi-boundary}
\end{equation}
\end{lemma}

Two holomorphic families with the same central map $f_0$ and the
same tangent field $\xi$ therefore have the same paired second variation
$H_{\R}+H_{i\R}$.

\begin{proof}
Set $\widetilde\omega\coloneqq \mathcal F^*\omega$, and write its Hermitian
coefficient matrix in the source coordinates $\zeta^1=z$, $\zeta^2=t$ as
\[
 \widetilde\omega=\frac{i}{2}\sum_{\alpha,\beta=1}^2
 \widetilde\omega_{\alpha\bar\beta}\,
 d\zeta^\alpha\wedge d\overline{\zeta^\beta}.
\]
The K\"ahler condition gives
$d\widetilde\omega=\mathcal F^*(d\omega)=0$. Closedness relates
derivatives in the parameter direction to derivatives in the source
direction. Comparing the $(2,1)$ and $(1,2)$ coefficients yields
\[
 \partial_t\widetilde\omega_{z\bar z}
 =\partial_z\widetilde\omega_{t\bar z},
 \qquad
 \partial_{\bar t}\widetilde\omega_{t\bar z}
 =\partial_{\bar z}\widetilde\omega_{t\bar t}.
\]
Differentiating these identities and commuting mixed derivatives gives
\[
 \partial_t\partial_{\bar t}\widetilde\omega_{z\bar z}
 =\partial_z\partial_{\bar t}\widetilde\omega_{t\bar z}
 =\partial_z\partial_{\bar z}\widetilde\omega_{t\bar t}.
\]
At $t=0$, the coefficient $\widetilde\omega_{t\bar t}$ is $|\xi|_\omega^2$.
Consequently
\[
 \partial_t\partial_{\bar t}\Area_R(f_t)\big|_{t=0}
 =\int_{\D_R}\partial_z\partial_{\bar z}|\xi|_\omega^2\,dA
 =\frac14\int_{\partial\D_R}\partial_\nu|\xi|_\omega^2\,ds.
\]
The sum of the second derivatives along the real and imaginary
parameter axes is the parameter Laplacian
$4\partial_t\partial_{\bar t}$. Hence
\[
 H_{\R}+H_{i\R}
 =4\left.\partial_t\partial_{\bar t}\Area_R(f_t)\right|_{t=0},
\]
which yields \eqref{eq:levi-boundary}.
\end{proof}

The classical Riemannian second-variation formula contains interior
curvature and second-fundamental-form terms; on K\"ahler submanifolds,
holomorphic normal fields are Jacobi fields~\cite[\S\S3.2, 3.5]{Simons1968}.
In our families the boundary values are allowed to move, so the first
variation of area need not vanish. Nevertheless, holomorphicity and
$d\omega=0$ give the exact paired boundary formula
\eqref{eq:levi-boundary}; the calculation also remains valid at branch
points.

Fix $R>0$ and a map $f$ holomorphic on a neighborhood of
$\overline\D_R$. To select variation fields before realizing them by
holomorphic families, define the boundary quadratic form on holomorphic
sections $\xi$ of $f^*T_X$ that are smooth on the closed disk:
\begin{equation}
 \mathcal H_{R,f}(\xi)\coloneqq
 \int_{\partial\D_R}\partial_\nu|\xi|_\omega^2\,ds.
 \label{eq:boundary-quadratic-form}
\end{equation}
When $\xi$ is the tangent field of a complex holomorphic family,
Lemma~\ref{lem:levi} identifies $\mathcal H_{R,f}(\xi)$ with the sum of the second area derivatives along the real
and imaginary parameter axes.

We use the Wiener--Masani boundary normalization in the form employed by
Berndtsson and Rosay~\cite[Theorem~2.1 and p.~888]{BerndtssonRosay2003};
Lemma~\ref{lem:factor} recalls the construction of the resulting
boundary-unitary holomorphic frame. The following computation combines
this normalization with Green's formula and the classical
determinant-curvature identities~\cite[Ch.~V, (13.3); Ch.~VIII, (6.6)--(6.7)]{DemaillyCADG2012}.
We record the matrix form that separates the Ricci integral from the
multiplier's Dirichlet cost and will give the descent threshold below.

\begin{proposition}\label{prop:matrixtrace}
Let $(X,\omega)$ be a Hermitian complex $n$-fold, let $R>0$, and let
$f:U\to X$ be holomorphic on an open neighborhood
$U\supset\overline\D_R$. Let
$e=(e_1,\ldots,e_n)$ be a holomorphic frame of $f^*T_X$, smooth
on the closed disk and unitary on $\partial\D_R$. For a positive
integer $N$, let $\Phi$ be a holomorphic
$\operatorname{Mat}_{n\times N}(\C)$-valued map on a neighborhood of
$\overline{\D}_R$, satisfying the boundary coisometry condition
\[
 \Phi(\zeta)\Phi(\zeta)^*=\Id_n
 \quad (|\zeta|=R).
\]
The columns $\Phi_a$ give sections
$(e\Phi)_a=\sum_{j=1}^ne_j\Phi_{ja}$ satisfying the trace identity
\begin{equation}
 \sum_{a=1}^N\mathcal H_{R,f}((e\Phi)_a)
 =4\int_{\D_R}\norm{\Phi'}_{\HS}^2\,dA
  -2\int_{\D_R}f^*\Ric(\omega).
 \label{eq:matrixtrace}
\end{equation}
\end{proposition}

Here $\Phi'$ denotes the entrywise derivative, and
$\|\cdot\|_{\HS}$ is the norm in \eqref{eq:matrix-norms}.

\begin{proof}
The Gram matrix $G\coloneqq(\langle e_j,e_k\rangle_\omega)_{j,k}$
represents squared norms by $|ec|_\omega^2=c^{\mathsf T}G\bar c$.
Thus
\[
 \sum_{a=1}^N |(e\Phi)_a|_\omega^2
 =\sum_{a=1}^N\Phi_a^{\mathsf T}G\overline{\Phi_a}
 =\tr(\Phi^{\mathsf T}G\bar\Phi)
 =\tr(G\bar\Phi\Phi^{\mathsf T}).
\]
The first trace sums the diagonal entries of $\Phi^{\mathsf T}G\bar\Phi$; the last
equality uses cyclicity of the trace. Since finite summation commutes
with the normal derivative and integration,
\eqref{eq:boundary-quadratic-form} gives the boundary integral of
$\partial_\nu\tr(G\bar\Phi\Phi^{\mathsf T})$.
On the boundary, $G=\bar\Phi\Phi^{\mathsf T}=\Id_n$, but their normal derivatives
need not vanish. The product rule and Jacobi's classical determinant formula
(see, for instance,~\cite{BhatiaJain2009}) give, on $\partial\D_R$,
\[
 \begin{aligned}
 \partial_\nu\tr(G\bar\Phi\Phi^{\mathsf T})
 &=\tr(\partial_\nu G)+\tr(\partial_\nu(\bar\Phi\Phi^{\mathsf T})),\\
 \partial_\nu\log\det G
 &=\tr(G^{-1}\partial_\nu G)=\tr(\partial_\nu G).
 \end{aligned}
\]
The boundary normalization is essential in the second line:
$\tr$ and $\log\det$ have the same first differential at $\Id_n$.
The divergence theorem and
$\Delta\tr(\bar\Phi\Phi^{\mathsf T})=4\|\Phi'\|_{\HS}^2$ give the multiplier term.
Since $\det G$ is the metric coefficient on $f^*K_X^{-1}$,
\eqref{eq:normalization} gives
\[
 f^*\Ric(\omega)=-\frac12\Delta\log\det G\,dA.
\]
Green's formula therefore turns the determinant term
$\partial_\nu\log\det G$ into
$-2\int f^*\Ric(\omega)$.
\end{proof}

We first construct variations that fix the value of a disk at $0$ and
one nonzero scalar component of its derivative. Fix $x_0\in X$ and
a nonzero complex-linear covector
\begin{equation}
  \lambda\in T_{x_0}^*X\setminus\{0\}.
  \label{eq:lambda}
\end{equation}
The scalar condition $\lambda(f_z(0))=1$ fixes one component of the
initial derivative and gives
$|f_z(0)|_\omega\geq\|\lambda\|_{\omega,*}^{-1}$, where the norm on the
covector is dual to $\omega$ at $x_0$.
The multiplier below imposes the linearized value and scalar-derivative
constraints. Proposition~\ref{cor:compact-fiber-descent} realizes these constraints
by holomorphic families that preserve $f(0)=x_0$ and
$\lambda(f_z(0))=1$ exactly, while leaving the other derivative
components free. In Section~\ref{sec:spread-sphere-quotient}, the center
varies in a compact fiber and the covector is pulled back from the base.
In the principal branch of Lemma~\ref{lem:relative-compactness-alternative},
this base normalization passes to the limit.
For a disk satisfying $f(0)=x_0$ and $\lambda(f_z(0))=1$, the frame
value $e(0):\C^n\to T_{x_0}X$ identifies the scalar constraint with the
nonzero covector $L_f\coloneqq\lambda\circ e(0)$. Choose
$U_f\in U(n)$ so that its last $n-1$ columns form an orthonormal basis
of $\ker L_f$. Fixing the center requires every column of the multiplier
to vanish at zero; in the last $n-1$ directions, which lie in $\ker L_f$, a simple
zero also preserves the scalar derivative to first order, while in the
remaining direction a double zero makes its derivative at zero
vanish. The resulting multiplier is
\begin{equation}
 \Phi_f(z)\coloneqq U_f\operatorname{diag}
 \left((z/R)^2,z/R,\ldots,z/R\right).
 \label{eq:scalar-multiplier}
\end{equation}
The columns of $\Phi_f$ vanish at zero, so
$(e\Phi_f)'(0)=e(0)\Phi_f'(0)$. The first column of $\Phi_f$ has
zero derivative at zero; the derivatives of its remaining columns
belong to $\ker L_f\subset\C^n$. Thus $\Phi_f(0)=0$ and
$L_f\Phi_f'(0)=0$ impose the infinitesimal constraints. On $|z|=R$, the diagonal entries have modulus
one, so $\Phi_f\Phi_f^*=\Id_n$.
Left multiplication by the unitary matrix $U_f$ preserves the
Hilbert--Schmidt norm. Differentiating the diagonal entries therefore
gives
\[
 \norm{\Phi_f'(z)}_{\HS}^2
 =\frac{4|z|^2}{R^4}+\frac{n-1}{R^2}.
\]
Integrating in polar coordinates, we obtain
\begin{equation}
 \int_{\D_R}\norm{\Phi_f'}_{\HS}^2dA
 =2\pi\int_0^R
   \left(\frac{4r^2}{R^4}+\frac{n-1}{R^2}\right)r\,dr
 =\pi(n+1).
 \label{eq:scalarcost}
\end{equation}
The Ricci lower bound for our chosen metric gives
\begin{equation}
 \sum_{a=1}^n\mathcal H_{R,f}((e\Phi_f)_a)
 \leq 4\pi(n+1)-2\kappa \Area_R(f).
 \label{eq:scalartrace}
\end{equation}

The cost in \eqref{eq:scalarcost} is optimal within this multiplier
class. For a nonzero covector $L\in(\C^n)^*$, suppose that $\Phi$ is
holomorphic near $\overline\D_R$, takes values in
$\operatorname{Mat}_{n\times N}(\C)$, and satisfies
\[
 \Phi\Phi^*=\Id_n\quad\hbox{on }\partial\D_R,
 \qquad \Phi(0)=0,
 \qquad L\Phi'(0)=0.
\]
Choose $Q\in U(n)$ with first row
$L/\|L\|$. Replacing $\Phi$ by $Q\Phi$ preserves the boundary
identity and the Hilbert--Schmidt Dirichlet integral. Its first row
vanishes to order at least two, while the remaining rows vanish to
order at least one.
Each row has norm one on $\partial\D_R$. For any such row
$r(z)=\sum_{k\geq m}a_k(z/R)^k$, Parseval's identity and radial
integration give
\[
 \sum_{k\geq m}\|a_k\|^2=1,
 \qquad
 \int_{\D_R}|r'|^2\,dA
 =\pi\sum_{k\geq m}k\|a_k\|^2\geq\pi m.
\]
Summing over the rows gives $2\pi+(n-1)\pi=\pi(n+1)$, attained by
the multiplier in \eqref{eq:scalar-multiplier}.

\section{Uniform holomorphic realization and disk operations}
\label{sec:analytic-engine}

We realize the trace calculation by holomorphic families and prove the
quantitative disk operations stated in
Propositions~\ref{cor:compact-fiber-descent}
and~\ref{prop:finite-area-enlargement}.
Holomorphic mapping spaces and Stein neighborhoods of graphs provide a
qualitative deformation theory~\cite{Forstneric2007}. Both the parameter
radius and the cubic area remainder must remain uniform as the disk and
its source radius vary. This control is obtained by lifting to a fixed
Stein submanifold of affine space, perturbing there, and projecting
back to $X$.

Assume in this section that $X$ is projective. Keep the projective
embedding $\iota:X\hookrightarrow\PP^{m-1}$ fixed as in
Section~\ref{sec:spread-sphere-quotient}; it will be used to represent
$\iota\circ f$ by homogeneous coordinates and to construct the lift space
$\mathcal Y$. The embedding $\iota$ is kept fixed as $f$ and
$R$ vary, so the constants below may depend on $\iota$.

We use the standard $L^p$ and $C^k$ norms on source regions, with
Euclidean area measure. Vector and operator norms are induced by the
fixed Hermitian and Euclidean metrics; matrix-valued $L^2$ estimates use
the Hilbert--Schmidt norm.

\subsection{Frame and lifting estimates}
\label{subsec:realization-estimates}

Bounds on area and weighted derivative do not control derivatives on the
source boundary. The following example shows that uniform realization
requires direct estimates. For
$0<r<1$ and $\varepsilon>0$, the functions $h_r:\D\to\C$ given by
\[
 h_r(z)\coloneqq\varepsilon\frac{1-r^2}{1-rz}
\]
are holomorphic near $\overline\D$. A direct computation gives
\[
 \int_\D|h_r'|^2dA=\pi\varepsilon^2r^2,
 \qquad B_1(h_r)\leq\frac{\varepsilon}{2},
 \qquad |h_r'(1)|=\varepsilon\frac{r(1+r)}{1-r}\longrightarrow\infty.
\]
Bounded area and bounded $B_1$ therefore do not imply relative compactness in
$C^1(\overline\D)$ (with the Euclidean metric on $\C$); the uniform
estimates below must be proved directly.

Boundary spectral factorization and a Green potential bound control a
normalized frame. Homogeneous coordinates and Nicolau's corona
theorem~\cite{Nicolau1990} then give a uniformly bounded affine lift
with bounded Dirichlet integral.

Fix $A_0,B_0\geq0$. In the auxiliary results below, $f$ is holomorphic near $\overline{\D}_R$
and satisfies $\Area_R(f)\leq A_0$ and
$B_R(f)\leq B_0$. Norms on fixed
finite-rank bundles are chosen once and for all.

The spectral factorization theorem is due to Wiener and Masani~\cite[Theorem~7.13]{WienerMasani1957}. We use the smooth boundary
version stated by Berndtsson and Rosay~\cite[Theorem~2.1]{BerndtssonRosay2003}. It makes the frame in
Theorem~\ref{thm:uniform-lift} unitary on the boundary, as required by
the Ricci-trace identity.

\begin{lemma}\label{lem:factor}
Let $H:\partial\D\to\operatorname{Mat}_{n\times n}(\C)$ be a smooth
positive-definite Hermitian matrix function. There is a map
$A:\D\to\operatorname{GL}_n(\C)$, holomorphic in $\D$ and smooth on
$\overline\D$ with invertible boundary values, such that
\[
 A(\zeta)^*H(\zeta)A(\zeta)=\Id_n
 \quad(|\zeta|=1).
\]
\end{lemma}

\begin{proof}
The Wiener--Masani theorem in the smooth boundary form of~\cite[Theorem~2.1]{BerndtssonRosay2003} gives
$H=h^*h$, where $h$ and $h^{-1}$ are holomorphic on $\D$ and smooth
on $\overline\D$. Thus $A\coloneqq h^{-1}$ has the required boundary
normalization.
\end{proof}

For $z\in\D$, let $\delta_z$ be the unit point mass at $z$. The Dirichlet
Green function of the unit disk, with the sign and normalization used
here, is~\cite[Ch.~I, \S4.B]{DemaillyCADG2012}
\[
 g(z,w)=\log\left|\frac{1-\bar zw}{z-w}\right|,
 \qquad -\Delta_wg=2\pi\delta_z.
\]
On $\D_R$ we use $g_R(z,w)\coloneqq g(z/R,w/R)$.

The Green potential controls both the frame and the homogeneous lift.
The following elementary estimate separates a neighborhood of the
logarithmic singularity, controlled by $B_0$, from its complement,
controlled by the total area.

\begin{lemma}\label{lem:green}
Let $(X,\omega)$ be a Hermitian complex manifold, let $R>0$ and
$A_0,B_0\geq0$, and let $f:U\to X$ be holomorphic on an open
neighborhood $U\supset\overline\D_R$. Let
$\varepsilon_f^\omega(w)\coloneqq |f_z(w)|_\omega^2$ be the area
density of $f$, so that $f^*\omega=\varepsilon_f^\omega\,dA$ and
$\Area_R(f)=\int_{\D_R}\varepsilon_f^\omega\,dA$.
If $\Area_R(f)\leq A_0$ and $B_R(f)\leq B_0$, then
\begin{equation}
 \sup_{z\in\D_R}\int_{\D_R}g_R(z,w)\varepsilon_f^\omega(w)dA(w)
 \leq C_g(A_0,B_0)
 \coloneqq C_G(B_0^2+A_0),
 \label{eq:greenbound}
\end{equation}
where $C_G$ is a universal constant, independent of $R$ and $f$.
\end{lemma}

\begin{proof}
First reduce to the unit disk. Under the dilation
$h(\zeta)\coloneqq f(R\zeta)$, the identities
\[
 \Area_1(h)=\Area_R(f),\qquad B_1(h)=B_R(f),\qquad
 \varepsilon_h^\omega(\eta)=R^2\varepsilon_f^\omega(R\eta)
\]
show that the hypotheses are unchanged. Since
$g_R(R\zeta,R\eta)=g(\zeta,\eta)$, change of variables also preserves
the Green-potential integral. We may therefore assume $R=1$.

Fix $z\in\D$ and set $\delta\coloneqq1-|z|$, its distance to the
boundary. Split the integral into the disk $|w-z|<\delta/2$ and its
complement in $\D$. On the first region,
\[
 1-|w|\geq1-|z|-|w-z|>\frac\delta2,
 \qquad
 |f_z(w)|_\omega\leq\frac{B_0}{1-|w|}.
\]
Thus $\varepsilon_f^\omega(w)\leq4B_0^2/\delta^2$.
To bound the Green kernel, write
\[
 1-\bar zw=(1-|z|^2)+\bar z(z-w).
\]
Since $1-|z|^2=(1+|z|)\delta\leq2\delta$, the triangle inequality gives
\[
 |1-\bar zw|\leq2\delta+|z-w|\leq\frac{5\delta}{2},
 \qquad
 g(z,w)\leq\log\frac{5\delta}{2|z-w|}.
\]
Integration in polar coordinates bounds this contribution by
\[
 \frac{4B_0^2}{\delta^2}\,2\pi
 \int_0^{\delta/2}r\log\frac{5\delta}{2r}\,dr
 =\pi B_0^2\left(\log5+\frac12\right).
\]

On the complementary region $|z-w|\geq\delta/2$, the same triangle
inequality instead yields
\[
 \frac{|1-\bar zw|}{|z-w|}
 \leq1+\frac{2\delta}{|z-w|}\leq5.
\]
Thus $g(z,w)\leq\log5$ there, and the remaining contribution is at most
\[
 (\log5)\int_\D\varepsilon_f^\omega\,dA
 \leq A_0\log5.
\]
Adding the two bounds proves \eqref{eq:greenbound}; one may take the
universal constant $C_G\coloneqq\pi(\log5+1/2)$, independently of $z$,
$R$, and $f$.
\end{proof}

We use the following Chern connection and curvature convention~\cite[Ch.~V, (12.2)--(12.4)]{DemaillyCADG2012}.
In the coordinates of \eqref{eq:normalization}, a section with
coefficient column $c$ has squared norm $c^{\mathsf T}h\bar c$.
Writing $\Theta(T_X)=\nabla^2$, the connection and curvature matrices are
\[
 \nabla'c=\partial c+\bar h^{-1}(\partial\bar h)c,
 \qquad \nabla''c=\bar\partial c,
 \qquad \Theta(T_X)=\bar\partial(\bar h^{-1}\partial\bar h).
\]
The coefficient of $dw^\alpha\wedge d\bar w^\beta$ in $\Theta(T_X)$ is
$-\partial_{\bar w^\beta}(\bar h^{-1}\partial_{w^\alpha}\bar h)$;
the minus sign comes from reversing the order of the differentials. For
$v=\sum_\alpha v^\alpha\partial_{w^\alpha}$ and a tangent vector $u$
with coefficient column $c$, our curvature tensor is
\begin{equation}
 \begin{split}
 R^\omega(v,\bar v,u,\bar u)
 \coloneqq\sum_{\alpha,\beta}v^\alpha\overline{v^\beta}\,
 c^{\mathsf T}\bigl[-\partial_{w^\alpha}\partial_{\bar w^\beta}h
 +(\partial_{w^\alpha}h)h^{-1}(\partial_{\bar w^\beta}h)\bigr]\bar c.
 \end{split}
 \label{eq:curvature-local-definition}
\end{equation}
Positive holomorphic bisectional curvature therefore means that
\eqref{eq:curvature-local-definition} is positive for every pair of
nonzero tangent vectors $v,u$ at the same point.
For a holomorphic disk $f$, the curvature endomorphism
$\mathcal R_f^\omega$ is characterized by
\begin{equation}
 \langle \mathcal R_f^\omega v,v\rangle_\omega
 =R^\omega(f_z,\overline{f_z},v,\bar v).
 \label{eq:curvature-endomorphism}
\end{equation}
The Bochner formula and the curvature formula for the determinant bundle~\cite[Ch.~VIII, proof of Lemma~7.2; Ch.~V, (4.6)]{DemaillyCADG2012}
give, for a holomorphic section $s$ of $f^*T_X$,
\begin{equation}
 \partial_z\partial_{\bar z}|s|_\omega^2
 =|\nabla_zs|_\omega^2-\langle \mathcal R_f^\omega s,s\rangle_\omega,
 \qquad
 f^*\Ric(\omega)=2\tr(\mathcal R_f^\omega)\,dA.
 \label{eq:bochner-ricci-convention}
\end{equation}
Set
\[
 C_\omega\coloneqq
 \max_{x\in X}
 \sup_{\substack{\xi,u\in T_xX\\
 |\xi|_\omega=|u|_\omega=1}}
 \left|R^\omega(\xi,\bar\xi,u,\bar u)\right|<\infty.
\]
Then
\[
 \|\mathcal R_f^\omega(z)\|_{\op}
 \leq C_\omega|f_z(z)|_\omega^2
 =C_\omega\varepsilon_f^\omega(z).
\]

For a frame $e=(e_1,\ldots,e_n)$, write
$G=(\langle e_j,e_k\rangle_\omega)$. Matrix inequalities refer to
Hermitian quadratic forms. The Green bound and Bochner identity control
$G$ throughout the disk and bound the frame's covariant Dirichlet integral.
Both estimates enter the lifted-field bounds in
Theorem~\ref{thm:uniform-lift}. For the two-sided metric control we apply
Berndtsson and Rosay's curvature-potential
estimate~\cite[Proposition~2.3]{BerndtssonRosay2003},
making the dependence on
$A_0,B_0$ explicit. The covariant Dirichlet bound then follows from
the integrated Bochner identity.

\begin{proposition}\label{prop:framebounds}
Let $(X,\omega)$ be a compact Hermitian complex $n$-fold, and fix
$A_0,B_0\geq0$. There are constants $\Gamma\geq1$ and $D_1<\infty$,
depending only on $(X,\omega,A_0,B_0)$, with the following property.
For any $R>0$ and any $f$ holomorphic on a neighborhood of
$\overline\D_R$ with $\Area_R(f)\leq A_0$ and $B_R(f)\leq B_0$,
every holomorphic frame $e$ of $f^*T_X$ that
is smooth on the closed disk and unitary on its boundary has Gram matrix
$G$ satisfying
\begin{equation}
 \Gamma^{-1}\Id_n\leq G\leq\Gamma\Id_n.
 \label{eq:Gamma}
\end{equation}
Moreover
\begin{equation}
 \sum_{a=1}^n\int_{\D_R}|\nabla_ze_a|^2dA
 \leq D_1.
 \label{eq:De}
\end{equation}
\end{proposition}

Both $\Gamma$ and $D_1$ are independent of $R$, $f$, and the
choice of boundary-unitary frame.

\begin{proof}
Set
\[
 P(z)=\int_{\D_R}g_R(z,w)\varepsilon_f^\omega(w)\,dA(w),\qquad
 \phi(z)=\frac{C_\omega}{\pi}C_g(A_0,B_0)
              -\frac{2C_\omega}{\pi}P(z).
\]
The Green normalization and Lemma~\ref{lem:green} give
\[
 \phi_{z\bar z}=C_\omega\varepsilon_f^\omega,
 \qquad \|\phi\|_\infty\leq
          \frac{C_\omega}{\pi}C_g(A_0,B_0).
\]
In the holomorphic frame $e$, the metric is the identity on the
boundary, and its curvature satisfies
$|\langle\mathcal R_f^\omega v,v\rangle_\omega|
 \leq\phi_{z\bar z}|v|_\omega^2$.
Berndtsson--Rosay's estimate~\cite[Proposition~2.3]{BerndtssonRosay2003}, applied after dilation to
the unit disk, therefore bounds the metric between
$e^{-2\|\phi\|_\infty}\Id_n$ and
$e^{2\|\phi\|_\infty}\Id_n$.
Both $\tr G$ and $\tr(G^{-1})$ are at most
\[
 \Gamma\coloneqq n\exp\!\left(
        \frac{2C_\omega C_g(A_0,B_0)}{\pi}\right),
\]
which gives \eqref{eq:Gamma}.

It remains to bound the frame's covariant Dirichlet integral.
We first sum the Bochner identity; after integration, the boundary
determinant calculation from Proposition~\ref{prop:matrixtrace} will
identify its left-hand side:
\[
 \partial_z\partial_{\bar z}\tr G
 =\sum_a|\nabla_ze_a|^2-\tr(\mathcal R_f^\omega\mathsf S),
 \qquad \mathsf S\coloneqq \sum_ae_a\otimes e_a^\flat.
\]
Here $e_a^\flat(v)=\langle v,e_a\rangle_\omega$ denotes metric duality.
Since $G=\Id_n$ on $\partial\D_R$, Jacobi's formula gives
\[
 \partial_\nu\log\det G
 =\tr(G^{-1}\partial_\nu G)=\partial_\nu\tr G
 \quad\hbox{on }\partial\D_R.
\]
Consequently Green's formula and \eqref{eq:normalization} give
\[
 \int_{\D_R}\partial_z\partial_{\bar z}\tr G\,dA
 =\frac14\int_{\partial\D_R}\partial_\nu\log\det G\,ds
 =-\frac12\int_{\D_R}f^*\Ric(\omega).
\]
Integrating the Bochner identity and solving for the nonnegative
derivative term therefore gives
\[
 \sum_a\int_{\D_R}|\nabla_ze_a|^2dA
 =\int_{\D_R}\tr(\mathcal R_f^\omega\mathsf S)dA
  -\frac12\int_{\D_R} f^*\Ric(\omega).
\]
Now $\tr\mathsf S\leq n\Gamma$, while the absolute Ricci density is at most
$2nC_{\omega}\varepsilon_f^\omega$, since
$f^*\Ric=2\tr(\mathcal R_f^\omega)\,dA$. Taking absolute values and using
$\Area_R(f)\leq A_0$ gives
\[
 \sum_a\int_{\D_R}|\nabla_ze_a|^2dA
 \leq nC_\omega(\Gamma+1)A_0.
\]
$D_1\coloneqq nC_\omega(\Gamma+1)A_0$ is therefore an admissible choice
in \eqref{eq:De}.
\end{proof}

Recall the projective embedding $\iota:X\hookrightarrow\PP^{m-1}$
chosen at the beginning of this section. Let $\omega_{\rm FS}$ be the
Fubini--Study form on its target, with normalization
\[
 \omega_{\rm FS}=\frac{i}{2}\partial\bar\partial\log\sum_{j=1}^m|Z_j|^2.
\]
Choose $c_\iota>0$ such that $\iota^*\omega_{\rm FS}\leq c_\iota\omega$;
such a constant exists by compactness of $X$. The displayed
Fubini--Study formula is read in homogeneous coordinates
$Z=(Z_1,\ldots,Z_m)\neq0$; it is unchanged as a form on projective
space when the homogeneous representative is multiplied by a
nonvanishing holomorphic function.

In rank one, boundary normalization subtracts a harmonic function from
the logarithm of the metric, as in Berndtsson and Rosay~\cite[p.~886]{BerndtssonRosay2003}. Together with the classical
Fubini--Study potential~\cite[Ch.~VI, (4.4)]{DemaillyCADG2012}, this
gives homogeneous coordinates with the bounds needed for the corona
theorem. We record the area identity and the uniform lower bound.

\begin{lemma}\label{lem:homlift}
Let $(X,\omega)$ be a compact projective manifold with a smooth
Hermitian metric, and fix an embedding
$\iota:X\hookrightarrow\PP^{m-1}$ and bounds $A_0,B_0\geq0$.
For every $R>0$, a map $f:\D_R\to X$ holomorphic on a neighborhood
of $\overline\D_R$ with $\Area_R(f)\leq A_0$ and
$B_R(f)\leq B_0$ has a homogeneous lift $\mathbf s$ from an open
neighborhood of $\overline{\D}_R$ to $\C^m\setminus\{0\}$, holomorphic on
that neighborhood, with $[\mathbf s]=\iota\circ f$ and
$|\mathbf s|=1$ on the boundary. It satisfies
\begin{equation}
 |\mathbf s|\leq1,
 \qquad
 \int_{\D_R}|\mathbf s'|^2dA
 =\int_{\D_R}f^*\iota^*\omega_{\rm FS}
 \leq c_\iota A_0,
 \label{eq:sdir}
\end{equation}
and
\begin{equation}
 |\mathbf s(z)|\geq\delta_{\mathbf s}>0,
 \label{eq:slower}
\end{equation}
where
$\delta_{\mathbf s}=\delta_{\mathbf s}
(X,\omega,\iota,A_0,B_0)$ is independent of $R$ and $f$.
\end{lemma}

\begin{proof}
The tautological line bundle on $\PP^{m-1}$ has fiber $\C Z$ over
$[Z]$. Trivialize its pullback by $\iota\circ f$ on a slightly larger
disk, using the Oka--Grauert principle~\cite[Satz~6]{Grauert1958}, and choose a nonvanishing section,
represented by a lift $v$ in
$\C^m\setminus\{0\}$. The positive real-analytic boundary function
$|v|$ has a logarithm whose harmonic extension is the real part of a
holomorphic function $a$ on a slightly larger disk. Then
$\mathbf s\coloneqq e^{-a}v$ has boundary norm one, and the maximum
principle gives $|\mathbf s|\leq1$.

The area density of $\iota\circ f$ for the Fubini--Study metric is
$\varepsilon_f^{\rm FS}\coloneqq |(\iota\circ f)'|_{\rm FS}^2$, so that
$f^*\iota^*\omega_{\rm FS}=\varepsilon_f^{\rm FS}\,dA$. Set
$u\coloneqq\log|\mathbf s|^2$. The boundary normalization $|\mathbf s|=1$ gives $u=0$ and
$\partial_\nu u=\partial_\nu|\mathbf s|^2$ there. Since $\mathbf s$
has no common zero, the Fubini--Study formula gives
$\Delta u=4\varepsilon_f^{\rm FS}$, while holomorphicity gives
$\Delta|\mathbf s|^2=4|\mathbf s'|^2$. Green's formula therefore yields
\[
 \int_{\D_R}|\mathbf s'|^2dA
 =\frac14\int_{\partial\D_R}\partial_\nu|\mathbf s|^2ds
 =\frac14\int_{\partial\D_R}\partial_\nu u\,ds
 =\int_{\D_R}\varepsilon_f^{\rm FS}dA,
\]
which proves \eqref{eq:sdir}. Green representation for $u$ with zero
boundary values~\cite[Ch.~I, (4.5)]{DemaillyCADG2012} gives
\[
 -u(z)=\frac2\pi\int g_R(z,w)\varepsilon_f^{\rm FS}(w)dA(w)
 \leq\frac{2c_\iota C_g(A_0,B_0)}\pi.
\]
\eqref{eq:slower} therefore holds with
$\delta_{\mathbf s}=\exp(-c_\iota C_g(A_0,B_0)/\pi)$.
\end{proof}

Write $H^\infty(\D)$ for the bounded holomorphic functions with the
supremum norm, and let
\[
 \mathcal D(\D)\coloneqq
 \left\{h:\D\to\C\text{ holomorphic}:\int_\D|h'|^2dA<\infty\right\}.
\]
The Dirichlet algebra is
\[
 \mathfrak A\coloneqq H^\infty(\D)\cap\mathcal D(\D),
 \qquad
 \|h\|_{\mathfrak A}\coloneqq \|h\|_\infty+
 \left(\int_\D|h'|^2dA\right)^{1/2}.
\]
To construct the affine lift, we need a holomorphic dual vector $\chi$
with $\chi\cdot\mathbf s=1$ and bounds on both its supremum norm and
Dirichlet integral. Nicolau's theorem and the bound at the beginning of
its proof~\cite[pp.~135--136]{Nicolau1990} give the following statement.

\begin{theorem}\label{thm:nicolau}
For every $m\geq1$ and $A,\delta>0$ there is
$C_N(m,A,\delta)<\infty$ such that, whenever
$u_1,\ldots,u_m\in\mathfrak A$ satisfy
\[
 \|u_j\|_{\mathfrak A}\leq A,
 \qquad
 \max_j|u_j(z)|\geq\delta\quad(z\in\D),
\]
there are $v_1,\ldots,v_m\in\mathfrak A$ with
\[
 \sum_ju_jv_j=1,
 \qquad
 \|v_j\|_{\mathfrak A}\leq C_N(m,A,\delta).
\]
The same statement holds on every disk after dilation.
\end{theorem}

\begin{proof}
The case $\delta>A$ is vacuous; assume $\delta\leq A$.
Nicolau uses the norm
\[
 \|h\|_{\mathrm N}\coloneqq\|h\|_\infty+
 \left(\frac1\pi\int_\D|h'|^2dA\right)^{1/2},
 \qquad
 \|h\|_{\mathrm N}\leq\|h\|_{\mathfrak A}
 \leq\sqrt\pi\,\|h\|_{\mathrm N}.
\]
His corona condition is the same maximum lower bound as above~\cite[p.~135, (1)]{Nicolau1990}.
The generators $u_j/A$ therefore have Nicolau norm at most $1$
and lower bound $\delta/A$. The estimate at the beginning of his proof~\cite[p.~136]{Nicolau1990} gives solutions $b_j$ with
$\sum_j(u_j/A)b_j=1$ and
$\|b_j\|_{\mathrm N}\leq K_m(\delta/A)$, where $K_m$ depends only on
the indicated data. Setting $v_j\coloneqq b_j/A$ gives
$\sum_j u_jv_j=1$ and
\[
 \|v_j\|_{\mathfrak A}\leq
 \frac{\sqrt\pi}{A}K_m(\delta/A),
\]
which is an admissible choice of $C_N(m,A,\delta)$.
For a disk of radius $R$,
\[
 \|h(R\,\cdot)\|_{\infty,\D}=\|h\|_{\infty,\D_R},\qquad
 \int_\D|(h(R\zeta))'|^2\,dA(\zeta)
 =\int_{\D_R}|h'(z)|^2\,dA(z),
\]
so the same bound applies after dilation.
\end{proof}

For a normalized lift $\mathbf s$ from Lemma~\ref{lem:homlift}, its
coordinates have algebra norm at most
$C_{\mathbf s}\coloneqq 1+\sqrt{c_\iota A_0}$, and at least one
has modulus at least $\delta_{\mathbf s}/\sqrt m$ at every point,
because $|\mathbf s|^2\leq m\max_j|\mathbf s_j|^2$.
Theorem~\ref{thm:nicolau} gives a holomorphic map
$\chi:\D_R\to(\C^m)^*$, whose components belong to the
algebra on $\D_R$ obtained from $\mathfrak A$ by dilation, satisfying
\begin{equation}
 \chi\cdot\mathbf s=1,
 \qquad
 \|\chi\|_\infty+\|\chi'\|_2\leq C_{\chi}
 \label{eq:coronamate}
\end{equation}
for a constant
$C_{\chi}=C_{\chi}
(X,\omega,\iota,A_0,B_0)$.
The dot denotes the bilinear pairing
$\chi\cdot\mathbf s=\sum_{j=1}^m\chi_j\mathbf s_j$, without complex
conjugation. The constant absorbs the fixed dimension factors in
passing from scalar to vector norms.

To extend the lift past $\overline\D_R$, dilate $\chi$ radially while
keeping $\mathbf s$ fixed, then restore their pairing by division.
For $0<\rho<1$ sufficiently close to $1$, set
\[
 \chi_\rho(z)\coloneqq\chi(\rho z),\qquad
 \gamma_\rho\coloneqq\chi_\rho\cdot\mathbf s,\qquad
 \widehat\chi\coloneqq\frac{\chi_\rho}{\gamma_\rho}.
\]
Since $\chi(\rho z)\cdot\mathbf s(\rho z)=1$, we have
\[
 \gamma_\rho(z)-1
 =\chi(\rho z)\cdot\bigl(\mathbf s(z)-\mathbf s(\rho z)\bigr).
\]
The fixed lift $\mathbf s$ is uniformly continuous on $\overline\D_R$,
and $\chi$ is bounded. So $\gamma_\rho\to1$ uniformly there as
$\rho\uparrow1$, and we may require $|\gamma_\rho|\geq1/2$.
The quotient is therefore holomorphic near $\overline\D_R$ and satisfies
$\widehat\chi\cdot\mathbf s=1$ and $\|\widehat\chi\|_\infty\leq2C_\chi$.
Dilation gives $\|\chi_\rho'\|_2\leq C_\chi$; the product and quotient rules give
\begin{equation}
 \begin{aligned}
 \|\gamma_\rho'\|_2&\leq C_\chi(1+\sqrt{c_\iota A_0}),\\
 \|\widehat\chi'\|_2
 &\leq2C_\chi+4C_\chi^2(1+\sqrt{c_\iota A_0})\eqqcolon D_\chi.
 \end{aligned}
 \label{eq:Dchi}
\end{equation}
All norms here are on $\D_R$. The bound depends only on the fixed target,
embedding, and $A_0,B_0$; the base lift $\mathbf s$ is unchanged.

\subsection{Uniform lifts and holomorphic variations}

The affine model below is a homogeneous version of the classical
incidence-complement construction known as the Jouanolou trick;
see~\cite[\S2.6, p.~771]{Forstneric2013Survey}.
The pairs $(\mathbf s,\chi)$ lie in this fixed model, whose defining
equation $\chi\cdot\mathbf s=1$ excludes the cone vertex, and a
holomorphic right inverse lifts tangent fields into it. The model
depends only on the embedding; the area and weighted-derivative caps
determine a compact subset and a retraction near it.

\begin{lemma}
\label{lem:fixed-affine-lift}
Let $X$ be a smooth projective manifold, and fix a projective embedding
$\iota:X\hookrightarrow\PP^{m-1}$. There are a closed Stein submanifold
$\mathcal Y\subset\C^{2m}$, a holomorphic submersion
$\pi_X:\mathcal Y\to X$, and a holomorphic bundle map
$\mathcal L:\pi_X^*T_X\to T_{\mathcal Y}$ such that
$d\pi_X\circ\mathcal L=\Id$, all depending only on $\iota$.
\end{lemma}

\begin{proof}
Consider the affine cone of the embedding,
\[
 \widehat X\coloneqq\{0\}\cup
 \{\mathbf s\in\C^m\setminus\{0\}:[\mathbf s]\in\iota(X)\},
\]
and set
\begin{equation}
 \mathcal Y\coloneqq
 \{(\mathbf s,\chi)\in\widehat X\times(\C^m)^*:
   \chi\cdot\mathbf s=1\},
 \qquad
 \pi_X(\mathbf s,\chi)\coloneqq \iota^{-1}[\mathbf s].
 \label{eq:lift-space}
\end{equation}
The defining equation excludes the cone vertex. On a chart
$\mathbf s_j\neq0$, local coordinates consist of coordinates on $X$, a
nonzero scale factor, and $m-1$ unconstrained dual coordinates. Hence
$\mathcal Y$ is a smooth closed affine algebraic variety and therefore
Stein, while $\pi_X$ is a holomorphic submersion.

For the holomorphic splitting, set
$V\coloneqq\C^m$ and let
$\ell\coloneqq\iota^*\mathcal O_{\PP^{m-1}}(-1)$ be the tautological
line bundle on $X$; its fiber $\ell_x\subset V$ is the line representing
$\iota(x)$. Define
\[
 \mathcal U\coloneqq
 \{(x,\chi)\in X\times V^*: \chi|_{\ell_x}\neq0\}.
\]
The map
\[
 \Theta:\mathcal Y\longrightarrow\mathcal U,
 \qquad (\mathbf s,\chi)\longmapsto(\pi_X(\mathbf s,\chi),\chi),
\]
is biholomorphic. Indeed, if $s_0$ is a local nonvanishing holomorphic
section of $\ell$, then
\[
 \sigma(x,\chi)\coloneqq\frac{s_0(x)}{\chi(s_0(x))}
\]
is independent of the choice of $s_0$, is holomorphic, lies in
$\ell_x$, and satisfies $\chi(\sigma(x,\chi))=1$. Thus
\[
 \Theta^{-1}(x,\chi)=(\sigma(x,\chi),\chi).
\]
Since $\mathcal U$ is open in $X\times V^*$, its tangent space at
$(x,\chi)$ is $T_xX\oplus V^*$. Transporting the vectors $(v,0)$ by
$d\Theta^{-1}$ gives the required holomorphic right inverse
\begin{equation}
 \mathcal L:\pi_X^*T_X\to T_{\mathcal Y},
 \qquad
 \mathcal L_y(v)\coloneqq d(\Theta^{-1})_{\Theta(y)}(v,0),
 \qquad d\pi_X\circ\mathcal L=\Id.
 \label{eq:L}
\end{equation}
\end{proof}

The holomorphic variations are built from three pieces of data: a
controlled boundary-unitary frame, a lift of the disk, and lifts of the
frame vectors. Under fixed area and weighted-derivative caps, the lifted
disks lie in a common compact subset of the fixed affine model, and
the estimates for all three ingredients are uniform.

\begin{theorem}
\label{thm:uniform-lift}
Let $(X,\omega)$ be a compact projective manifold of complex dimension
$n\geq1$ with a smooth
Hermitian metric, and fix a projective embedding
$\iota:X\hookrightarrow\PP^{m-1}$.
Fix the affine model $\mathcal Y$, the submersion $\pi_X$, and the
right inverse $\mathcal L$ supplied by
Lemma~\ref{lem:fixed-affine-lift}. Given $A_0,B_0\geq0$, one can choose a compact set
$K_{A_0,B_0}\Subset\mathcal Y$, an open
neighborhood $\Omega$ of $K_{A_0,B_0}$ in $\C^{2m}$,
a holomorphic map $\widetilde\pi_X:\Omega\to X$ with
$\widetilde\pi_X|_{\mathcal Y\cap\Omega}=\pi_X$, and constants
\[
 \Gamma,D_1,D_{\widetilde f},C_{\theta,\infty},D_\theta<\infty.
\]
These cap-dependent choices depend only on
$(X,\omega,\iota,A_0,B_0)$ and are independent of the
source radius and the individual disk; the affine model and its
splitting depend only on $\iota$. For every $R>0$ and every map $f$
holomorphic on a neighborhood of $\overline{\D}_R$ with
$\Area_R(f)\leq A_0$ and $B_R(f)\leq B_0$, one can choose a frame
$e$, a lifted disk $\widetilde f$, and a matrix $\theta$ compatibly so
that the following properties hold:
\begin{enumerate}
\renewcommand{\labelenumi}{\textup{(\roman{enumi})}}
\item $e$ is a holomorphic frame of $f^*T_X$, smooth on the closed disk
and unitary on its boundary, whose Gram matrix
$G(z)\coloneqq(\langle e_j(z),e_k(z)\rangle_\omega)_{j,k=1}^n$ satisfies
\[
 \Gamma^{-1}\Id_n\leq G\leq\Gamma\Id_n,
 \qquad
 \sum_{a=1}^n\int_{\D_R}|\nabla_ze_a|^2dA\leq D_1;
\]
\item $\widetilde f$ is a map from a neighborhood of
$\overline{\D}_R$ to $\mathcal Y$, holomorphic on that
neighborhood, such that
\[
 \pi_X\circ\widetilde f=f,
 \quad \widetilde f(\overline{\D}_R)
       \subset K_{A_0,B_0},
 \quad \int_{\D_R}|\widetilde f'|^2dA\leq D_{\widetilde f};
\]
\item the frame vectors have holomorphic lifts
$\theta_1,\ldots,\theta_n:\D_R\to\C^{2m}$ along $\widetilde f$.
Their column matrix and the lifting identities are
\begin{equation}
 \begin{gathered}
 \theta\coloneqq(\theta_1,\ldots,\theta_n):
 \D_R\to\operatorname{Mat}_{2m\times n}(\C),\\
 d(\widetilde\pi_X)_{\widetilde f(z)}\theta_a(z)=e_a(z)
 \qquad(1\leq a\leq n).
 \end{gathered}
 \label{eq:theta-lift-definition}
\end{equation}
The matrix $\theta$ is smooth on the closed disk and satisfies
\[
 \sup_{z\in\D_R}\norm{\theta(z)}_{\op}\leq C_{\theta,\infty},
 \qquad \int_{\D_R}\norm{\theta'}_{\HS}^2dA\leq D_\theta.
\]
\end{enumerate}
\end{theorem}

\begin{proof}
Fix a disk $f$ satisfying the caps $A_0,B_0$.

\par\smallskip\noindent\textup{(i)} Choose $R'>R$ so that $\overline{\D}_{R'}$ lies in the given source
neighborhood, and choose a holomorphic frame $\widetilde e$ of $f^*T_X$
on $\D_{R'}$; such a frame exists
by the Oka--Grauert principle~\cite[Satz~6, p.~270]{Grauert1958}, since the source is
contractible and Stein. Restrict its Gram matrix
to $\partial\D_R$ and apply Lemma~\ref{lem:factor} to
$H(\zeta)=\overline{\widetilde G(R\zeta)}$.
For the resulting factor $A$, set
$e(z)\coloneqq\widetilde e(z)A(z/R)$. Conjugating the factorization
identity gives its boundary Gram matrix
$A(z/R)^{\mathsf T}\widetilde G(z)\overline{A(z/R)}=\Id_n$.
$e$ is therefore holomorphic, smooth
on the closed disk, and unitary on its boundary.
Proposition~\ref{prop:framebounds} gives
\eqref{eq:Gamma} and \eqref{eq:De}, proving \textup{(i)}.

\par\smallskip\noindent\textup{(ii)} By Lemma~\ref{lem:homlift}, $f$ has a homogeneous lift $\mathbf s$ with
\[
 |\mathbf s|\leq1,\qquad |\mathbf s|\geq\delta_{\mathbf s},\qquad
 \int_{\D_R}|\mathbf s'|^2dA\leq c_\iota A_0.
\]
Its coordinates satisfy the hypotheses of Theorem~\ref{thm:nicolau}
with algebra-norm bound $C_{\mathbf s}$ and corona lower bound
$\delta_{\mathbf s}/\sqrt m$. Thus \eqref{eq:coronamate} supplies a vector $\chi$ satisfying $\chi\cdot\mathbf s=1$, and the dilation construction leading to \eqref{eq:Dchi} gives a vector
$\widehat\chi$, holomorphic
near the closed disk, with
\[
 \widehat\chi\cdot\mathbf s=1,\qquad
 \|\widehat\chi\|_{L^\infty(\D_R)}\leq2C_{\chi},\qquad
 \|\widehat\chi'\|_{L^2(\D_R)}\leq D_\chi.
\]

The identity $\widehat\chi\cdot\mathbf s=1$ shows that
$\widetilde f\coloneqq(\mathbf s,\widehat\chi)$ takes values in the
fixed space $\mathcal Y$ of \eqref{eq:lift-space}. It satisfies
\begin{equation}
 \widetilde f(\overline{\D}_R)\subset
 K_{A_0,B_0}\coloneqq
 \mathcal Y\cap
 \{|\mathbf s|\leq1,\ |\chi|\leq2C_{\chi}\},
 \qquad
 \int_{\D_R}|\widetilde f'|^2dA
 \leq c_\iota A_0+D_\chi^2\eqqcolon D_{\widetilde f}.
 \label{eq:lifted-disk-bound}
\end{equation}
The set $K_{A_0,B_0}$ is compact: it is closed and
bounded, and $\chi\cdot\mathbf s=1$ prevents $\mathbf s$ from approaching
the cone vertex.

Since $\mathcal Y$ is a closed Stein submanifold of $\C^{2m}$,
the Docquier--Grauert tubular neighborhood theorem~\cite[Satz~3]{DocquierGrauert1960}, in the holomorphic
deformation-retraction formulation~\cite[Theorem~3.1]{Forstneric2010Stratified},
gives a holomorphic retraction $\mathfrak r:\Omega\to\mathcal Y\cap\Omega$
from a neighborhood $\Omega$ of $K_{A_0,B_0}$, equal to the identity on
$\mathcal Y\cap\Omega$. Its derivatives are bounded on smaller compact
neighborhoods of $K_{A_0,B_0}$. Set
$\widetilde\pi_X\coloneqq\pi_X\circ\mathfrak r$. Since $\pi_X\circ\widetilde f=f$ and
$\widetilde\pi_X|_{\mathcal Y\cap\Omega}=\pi_X$,
\eqref{eq:lifted-disk-bound} proves \textup{(ii)}. The choices of
$K_{A_0,B_0}$, $\Omega$, and $\widetilde\pi_X$ depend only on the
fixed target data and the two caps; they are fixed before any individual
lifted disk $\widetilde f$ is chosen. The space $\mathcal Y$, the
projection $\pi_X$, and the splitting $\mathcal L$ were already fixed
from the embedding alone.

\par\smallskip\noindent\textup{(iii)} Regard the splitting $\mathcal L$ from
\eqref{eq:L} as a bundle map from $\pi_X^*T_X$ to the ambient trivial bundle
$\mathcal Y\times\C^{2m}$. In $\nabla\mathcal L$, use the
pullback Chern connection on $\pi_X^*T_X$ and the flat connection on this
ambient bundle. Let
\[
 L_0\coloneqq \sup_{K_{A_0,B_0}}\|\mathcal L\|,
 \qquad
 L_1\coloneqq \sup_{K_{A_0,B_0}}\|\nabla\mathcal L\|.
\]
Set
\[
 \theta_a(z)\coloneqq \mathcal L_{\widetilde f(z)}e_a(z).
\]
Then $d(\widetilde\pi_X)_{\widetilde f}\theta_a=e_a$. With the preceding
connection convention,
ambient differentiation gives
\begin{equation}
 \theta_a'
 =\mathcal L_{\widetilde f}\nabla_ze_a
  +(\nabla_{\widetilde f'}\mathcal L)e_a.
 \label{eq:thetaderivative}
\end{equation}
The frame and lifted-disk bounds, with
$|a+b|^2\leq2|a|^2+2|b|^2$, give
\begin{align}
 \|\theta\|_{L^\infty,\op}
 &\leq L_0\sqrt\Gamma\eqqcolon C_{\theta,\infty},
 \label{eq:theta-sup-bound}\\
 \int_{\D_R}\|\theta'\|_{\HS}^2dA
 &\leq2L_0^2D_1+2n\Gamma L_1^2D_{\widetilde f}\eqqcolon D_\theta.
 \label{eq:theta-dirichlet-bound}
\end{align}
This proves \textup{(iii)}.
\end{proof}

We now add the lifted variation fields and project to $X$. The lifted
frame matrix $\theta$ is holomorphic on $\D_R$ but need not extend
holomorphically across $\partial\D_R$. For $0<\rho<1$, we therefore
replace $\theta(z)$ in the variation field by $\theta(\rho z)$. The new
matrix is holomorphic on the larger disk $\D_{R/\rho}$, so the resulting
field extends past $\overline\D_R$. We leave the central lift
$\widetilde f(z)$ unchanged; hence the projected family at parameter
$t=0$ is still the original map $f$. The multiplier is also unchanged,
so its vanishing orders at the origin are preserved. As $\rho\uparrow1$,
the regularized fields converge to the original fields, and their paired
trace converges to the value computed in
Section~\ref{sec:trace-lifting}. The parameter range and cubic remainder
remain uniform throughout this approximation.

The embedding fixes $\mathcal Y,\pi_X,\mathcal L$, independently of
the disks and caps, and the caps $A_0,B_0$ then fix
$K_{A_0,B_0},\Omega,\widetilde\pi_X$; these data, together with
$D_\Phi,N_\Phi$, determine $t_0,C_T$. Each disk comes with its own
compatible frame, lift, and lifted fields, as in
Theorem~\ref{thm:uniform-lift}. The approximation parameter $\rho$ may
depend on this triple, on $\Phi$, and on the prescribed trace error
$\varepsilon_{\mathrm{tr}}$, whereas $t_0$ and $C_T$ are independent
of $R,f,\Phi,\rho$, and $\varepsilon_{\mathrm{tr}}$; the extension
neighborhood may depend on the family, on $\Phi$, and on $\rho$, and
its width need not be uniform.

\begin{theorem}
\label{thm:engine}
Let $X$ be a compact projective manifold of dimension $n\geq1$,
equipped with a K\"ahler metric $\omega$. Fix the projective embedding
$\iota:X\hookrightarrow\PP^{m-1}$, and the bounds
$A_0,B_0\geq0$ as in
Theorem~\ref{thm:uniform-lift}. Let $D_\Phi\geq0$, and let
$N_\Phi\geq1$ be an integer. There are constants $t_0>0$ and
$C_T<\infty$, depending only on
$(X,\omega,\iota,A_0,B_0,D_\Phi,N_\Phi)$, with the following
property.

Let $R>0$, and let $f:\D_R\to X$ extend holomorphically to a
neighborhood of $\overline{\D}_R$ and satisfy
$\Area_R(f)\leq A_0$ and $B_R(f)\leq B_0$. Use the fixed
$K_{A_0,B_0},\Omega$, and $\widetilde\pi_X$ from
Theorem~\ref{thm:uniform-lift}, and choose any compatible disk-dependent
triple $e,\widetilde f,\theta$ supplied by that theorem. Suppose that
$\Phi=(\Phi_1,\ldots,\Phi_{N_\Phi}):\D_R\to
\operatorname{Mat}_{n\times N_\Phi}(\C)$ extends holomorphically to an open
neighborhood of $\overline{\D}_R$ and satisfies
\[
 \sup_{z\in\overline{\D}_R}\norm{\Phi(z)}_{\op}\leq1,
 \qquad
 \int_{\D_R}\norm{\Phi'}_{\HS}^2dA\leq D_\Phi,
\]
where the column $\Phi_a$ determines the section
$e\Phi_a=\sum_{j=1}^ne_j\Phi_{ja}$ of $f^*T_X$.
The constants $t_0$ and $C_T$ are independent of $R,f,\Phi$ and of this
compatible triple.

For each $0<\rho<1$, write
\[
 \Xi_{\rho}(z)\coloneqq
 \theta(\rho z)\Phi(z),\qquad
 (\Xi_{\rho})_a(z)\coloneqq
 \sum_{j=1}^n\theta_j(\rho z)\Phi_{ja}(z).
\]
\begin{enumerate}
\renewcommand{\labelenumi}{\textup{(\roman{enumi})}}
\item \emph{Holomorphic realization.} For every $0<\rho<1$, each
column defines the family
\[
 f^{\rho}_{a,t}(z)\coloneqq
 \widetilde\pi_X\bigl(\widetilde f(z)
              +t(\Xi_{\rho})_a(z)\bigr),
 \qquad t\in\C,\quad |t|\leq t_0,\quad 1\leq a\leq N_\Phi.
\]
All these families are jointly holomorphic on one open
neighborhood of
$\overline{\D}_R\times\{t\in\C:|t|\leq t_0\}$; for every
$0<\rho<1$ that neighborhood may
depend on the selected disk-dependent triple, $\Phi$, and $\rho$, but
not on the column $a$.
\item \emph{Uniform cubic remainder.} For every $0<\rho<1$,
$1\leq a\leq N_\Phi$, and $\tau\in\{1,i\}$, write
$t=\tau s$ with $s\in\R$ and set
\[
 \ell^{\rho}_{a,\tau}\coloneqq
 \left.\frac{d}{ds}\right|_{s=0}\Area_R(f^{\rho}_{a,\tau s}),
 \qquad
 H^{\rho}_{a,\tau}\coloneqq
 \left.\frac{d^2}{ds^2}\right|_{s=0}
 \Area_R(f^{\rho}_{a,\tau s}).
\]
For $|s|\leq t_0$,
\begin{equation}
 \Area_R(f^{\rho}_{a,\tau s})
 \leq \Area_R(f)+\ell^{\rho}_{a,\tau}s
       +\frac12H^{\rho}_{a,\tau}s^2
       +C_T|s|^3
 \label{eq:Taylor}
\end{equation}
with the same constant $C_T$ for both choices of $\tau$.
\item \emph{Trace approximation.} If $\Phi$ is boundary
coisometric, meaning $\Phi\Phi^*=\Id_n$ on $\partial\D_R$, then for every
$\varepsilon_{\mathrm{tr}}>0$ there is a number
\[
 \rho_{\mathrm{tr}}
 =\rho_{\mathrm{tr}}(R,f,e,\widetilde f,\theta,\Phi,
     \varepsilon_{\mathrm{tr}})\in(0,1)
\]
such that
\begin{equation}
 \left|\sum_{a=1}^{N_\Phi}
       \bigl(H^{\rho_{\mathrm{tr}}}_{a,1}
              +H^{\rho_{\mathrm{tr}}}_{a,i}\bigr)
 -\left(4\int_{\D_R}\|\Phi'\|_{\HS}^2dA
 -2\int_{\D_R}f^*\Ric(\omega)\right)\right|
 \leq\varepsilon_{\mathrm{tr}}.
 \label{eq:actual-paired-trace}
\end{equation}
\end{enumerate}
\end{theorem}

Only the trace approximation requires choosing $\rho_{\mathrm{tr}}$
close to $1$; the constants $t_0$ and $C_T$ remain independent of
$\varepsilon_{\mathrm{tr}}$ and $\rho_{\mathrm{tr}}$.

\begin{proof}
Use the fixed cap-dependent ambient data and the selected compatible
triple from Theorem~\ref{thm:uniform-lift}. Set
\[
 \delta_0\coloneqq
 \min\{1,\dist(K_{A_0,B_0},\C^{2m}\setminus\Omega)\}>0,
\]
and choose a compact $K'\Subset\Omega$ whose interior contains the closed
Euclidean $3\delta_0/4$-thickening of
$K_{A_0,B_0}$. In the standard coordinate frame of
$\C^{2m}$, let $g_0$ be the positive-semidefinite Hermitian
coefficient matrix of $\widetilde\pi_X^*\omega$ on $K'$, so
\[
 v^{\mathsf T}g_0(w)\bar v=|d(\widetilde\pi_X)_w(v)|_\omega^2
\]
for a coefficient column $v$. Only smooth coefficient bounds are used;
the pullback form may be degenerate along the fibers. Take
$C_{g,3}$ to bound the coefficients and their derivatives
through order three on $K'$.

Take a multiplier $\Phi$ satisfying the stated
bounds and, for $0<\rho<1$, set
$\Xi_{\rho}\coloneqq
\theta(\rho\,\cdot)\Phi$.
Dilation gives
$\int_{\D_R}|\rho\theta'(\rho z)|^2\,dA(z)
=\int_{\D_{\rho R}}|\theta'(w)|^2\,dA(w)\leq D_\theta$.
Use this identity, $\|\Phi\|_\infty\leq1$, and
$(\theta(\rho z)\Phi)'=
\rho\theta'(\rho z)\Phi+
\theta(\rho z)\Phi'$ to obtain
\begin{equation}
 \sup_{z\in\D_R}\|\Xi_{\rho}(z)\|_{\op}
 \leq C_{\theta,\infty},
 \qquad
 \int_{\D_R}\|\Xi_{\rho}'\|_{\HS}^2dA
 \leq2D_\theta+2C_{\theta,\infty}^2D_\Phi\eqqcolon D_\Xi.
 \label{eq:increment-bound}
\end{equation}
Choose
\[
 t_0\coloneqq
 \min\left\{1,
 \frac{\delta_0}{4\max\{1,C_{\theta,\infty}\}}\right\}.
\]
For $|t|\leq2t_0$, the displacement of each column perturbation
is at most $2t_0C_{\theta,\infty}\leq\delta_0/2$.
So $\widetilde f+t(\Xi_{\rho})_a$ stays in $K'$,
uniformly in the disk and the column.

\par\smallskip\noindent\textup{(i)}
Fix a column and write $\Xi\coloneqq(\Xi_{\rho})_a$. The function
$\theta(\rho z)$ is holomorphic for $|z|<R/\rho$; since
$\widetilde f$ and $\Phi$ extend past the closed disk, so does
$\widetilde f+t\Xi$. On the closed disk and for
$|t|\leq t_0$ its image lies in the closed
$\delta_0/2$-thickening of $K_{A_0,B_0}$, which
is compactly contained in $\operatorname{int}K'\Subset\Omega$.
Compactness gives a neighborhood of
$\overline\D_R\times\{|t|\leq t_0\}$ on which the image remains in
$\Omega$, so the projected family is jointly holomorphic there.
Intersect these neighborhoods over the finitely many columns.

\par\smallskip\noindent\textup{(ii)}
In ambient coordinates write $w(t)\coloneqq\widetilde f+t\Xi$ and
$v(t)\coloneqq\widetilde f'+t\Xi'$. The area density is
$v(t)^{\mathsf T}g_0(w(t))\overline{v(t)}$. Along either real parameter axis
$t=\tau s$, with $\tau\in\{1,i\}$, the functions $w(\tau s)$ and
$v(\tau s)$ are affine in $s$, with first derivatives $\tau\Xi$ and
$\tau\Xi'$. For fixed $z$, put
$G(s)\coloneqq g_0(w(\tau s))$, $V(s)\coloneqq v(\tau s)$, and
$Q\coloneqq\tau\Xi'$. Ordinary differentiation in the real variable $s$
gives
\[
 \frac{d^3}{ds^3}(V^{\mathsf T}G\bar V)
 =V^{\mathsf T}G'''\bar V
  +3Q^{\mathsf T}G''\bar V+3V^{\mathsf T}G''\bar Q+6Q^{\mathsf T}G'\bar Q.
\]
The chain rule bounds $|G^{(k)}(s)|$ by
$C_{g,3}|\Xi|^k$ for $1\leq k\leq3$, up to fixed dimension factors.
The bounds
$\|\Xi\|_\infty\leq C_{\theta,\infty}$ and $|t|\leq t_0$ consequently give the
integrable estimate
\[
 \left|\frac{d^3}{ds^3}
 \bigl(v(\tau s)^{\mathsf T}g_0(w(\tau s))\overline{v(\tau s)}\bigr)\right|
 \leq C(C_{g,3},C_{\theta,\infty},t_0)
       (|\widetilde f'|^2+|\Xi'|^2).
\]
The right-hand side has a uniformly bounded integral by
\eqref{eq:lifted-disk-bound} and \eqref{eq:increment-bound}.
The third derivative of the area is therefore uniformly bounded along
both real parameter axes. Taylor's theorem gives \eqref{eq:Taylor}
with one remainder constant $C_T$.

\par\smallskip\noindent\textup{(iii)}
The tangent fields of the actual families satisfy, for each fixed $f$,
\[
 \xi_{\rho,a}\coloneqq
 d(\widetilde\pi_X)_{\widetilde f}
 \bigl(\theta(\rho\,\cdot)\Phi_a\bigr)
 \longrightarrow e\Phi_a
 \quad\text{in }C^1(\overline{\D}_R)
 \quad(\rho\uparrow1),
\]
because $\theta(\rho\,\cdot)\to\theta$ in the same norm.
The K\"ahler condition enters at this point, through Lemma~\ref{lem:levi}.
If $\Phi$ is boundary coisometric, that lemma, continuity of
\eqref{eq:boundary-quadratic-form}, and Proposition~\ref{prop:matrixtrace}
therefore give
\[
 \sum_a\mathcal H_{R,f}(\xi_{\rho,a})
 \longrightarrow
 4\int_{\D_R}\|\Phi'\|_{\HS}^2\,dA
 -2\int_{\D_R}f^*\Ric(\omega).
\]
The summand is $H^\rho_{a,1}+H^\rho_{a,i}$ by Lemma~\ref{lem:levi}.
Choose $\rho_{\mathrm{tr}}$ close enough to one to obtain the
prescribed error. Parts \textup{(i)}--\textup{(ii)} and
\eqref{eq:increment-bound} hold for every $\rho\in(0,1)$ with the same
$t_0,C_T$. Their dependence is only on the fixed target data and
$A_0,B_0,D_\Phi,N_\Phi$.
\end{proof}

\subsection{A common area and derivative estimate}\label{eff:subsec:joint-estimate}

For the fixed caps, write $K=K_{A_0,B_0}$, $\Pi=\widetilde\pi_X$, and
$D=D_{\widetilde f}$ for the common compact set, ambient projection,
and Dirichlet bound supplied by Theorem~\ref{thm:uniform-lift}.
Each compatible lift $H$ satisfies
\begin{equation}\label{eq:quantitative-lift}
 \Pi H=f,\quad H(\overline\D_R)\subset K,\quad
 \int_{\D_R}|H'|^2dA\le D.
\end{equation}
Choose $d_K>0$ such that the closed Euclidean
$d_K$-neighborhood $K_{d_K}$ of $K$ is compactly contained in
$\Omega$. Write $g(w)$ for the Hermitian coefficient matrix of
$\Pi^*\omega$ in ambient coordinates. Thus
\[
 a^{\mathsf T}g(w)\bar a=|d\Pi_w(a)|_\omega^2.
\]
The form can be degenerate. On $K_{d_K}$, fix bounds
\begin{equation}\label{eff:eq:metric-bounds}
 \|g(w)\|_{\op}\leq M_0,\qquad
 \|Dg(w)[v]\|_{\op}\leq L_g|v|,
\end{equation}
where $Dg$ is the real derivative in ambient coordinates. These bounds
are fixed before selecting an individual disk.

\begin{lemma}\label{eff:lem:weighted-euclidean}
Let $R>0$, $N\geq1$, and $E\geq0$. If $F:\D_R\to\C^N$ is holomorphic and
$\int_{\D_R}|F'|^2dA\leq E$, then
\[
 (R-|z|)|F'(z)|\leq\sqrt{E/\pi}\quad(z\in\D_R).
\]
\end{lemma}
\begin{proof}
For $0<r<R-|z|$, the classical submean inequality~\cite[Ch.~I, \S4]{DemaillyCADG2012}, applied to $|F'|^2$, gives
\[
 \pi r^2|F'(z)|^2\leq\int_{\D_r(z)}|F'|^2dA\leq E.
\]
Let $r\uparrow R-|z|$.
\end{proof}

\begin{proposition}\label{eff:prop:joint-estimate}
Let $(X,\omega)$ be a compact projective manifold with a smooth
Hermitian metric. Fix $R>0$ and the cap-dependent ambient data
$K,\Omega,\Pi,D,d_K,M_0,L_g$ from
\eqref{eq:quantitative-lift}--\eqref{eff:eq:metric-bounds}.
Let $H$ be holomorphic near $\overline\D_R$ and satisfy
\eqref{eq:quantitative-lift}, with $f=\Pi\circ H$. Let $S,E\geq0$, and
let $\Xi$ be a $\C^{2m}$-valued map holomorphic near
$\overline\D_R$, with
\[
 \|\Xi\|_\infty\leq S,\qquad
 \int_{\D_R}|\Xi'|^2\,dA\leq E.
\]
Fix $0<\bar t\leq1$ with $\bar t S<d_K$.
For $|t|\leq\bar t$, set
$f_t=\Pi(H+t\Xi)$ and
\begin{equation}\label{eff:eq:general-cost}
 C_A=L_gSD+2M_0\sqrt{DE}+M_0\bar t E.
\end{equation}
Then
\begin{align}
 |\Area_R(f_t)-\Area_R(f)|&\leq C_A|t|,\label{eff:eq:general-A}\\
 |B_R(f_t)^2-B_R(f)^2|&\leq\frac{C_A}{\pi}|t|.
 \label{eff:eq:general-B}
\end{align}
The same family has the pointwise displacement bound
\begin{equation}\label{eff:eq:position-step}
 \dist_\omega(f_t(z),f(z))\le\sqrt{M_0}\,S|t|\qquad(z\in\D_R).
\end{equation}
\end{proposition}

The estimate for $B_R^2$ gives both an upper and a lower
bound for the new value of $B_R$:
\begin{equation}\label{eff:eq:B-interval}
 \sqrt{\max\{0,B_R(f)^2-C_A|t|/\pi\}}
 \leq B_R(f_t)
 \leq\sqrt{B_R(f)^2+C_A|t|/\pi}.
\end{equation}
\begin{proof}
The entire segment $H(z)+st\Xi(z)$, $0\leq s\leq1$, lies in
$K_{d_K}$, so \eqref{eff:eq:metric-bounds} applies along it. Put
$a=H'(z)$, $b=\Xi'(z)$, and $g_t=g(H(z)+t\Xi(z))$. The change in area density has three contributions: the change in
the metric coefficients, a mixed derivative term, and a quadratic
derivative term. More precisely,
\begin{equation}\label{eff:eq:density-identity}
 (a+tb)^{\mathsf T}g_t\overline{a+tb}-a^{\mathsf T}g_0\bar a
 =a^{\mathsf T}(g_t-g_0)\bar a+2\operatorname{Re}(t b^{\mathsf T}g_t\bar a)
       +|t|^2b^{\mathsf T}g_t\bar b.
\end{equation}
Its absolute value is at most
\[
 |t|L_gS|a|^2+2M_0|t||a||b|+M_0|t|^2|b|^2.
\]
Integrating, and applying Cauchy--Schwarz to the middle term, proves
\eqref{eff:eq:general-A}.

For the second estimate, put $d=R-|z|$. Lemma~\ref{eff:lem:weighted-euclidean}
gives $d|a|\leq\sqrt{D/\pi}$ and $d|b|\leq\sqrt{E/\pi}$.
Multiplying the same pointwise estimate by $d^2$ therefore bounds the
difference of the two squared weighted derivatives by $C_A|t|/\pi$.
Use $|\sup p-\sup q|\leq\sup|p-q|$ and
$B_R(f)^2=\sup d^2|f_z|_\omega^2$ to obtain \eqref{eff:eq:general-B}.
Finally, the path $s\mapsto\Pi(H(z)+st\Xi(z))$, $0\le s\le1$,
has speed at most $\sqrt{M_0}|t||\Xi(z)|$. Its length proves
\eqref{eff:eq:position-step}. Only bounds on the possibly degenerate
metric coefficients have been used.
\end{proof}

\subsection{Uniform descent and enlargement}
\label{subsec:disk-operations}

Return now to the relative setting, where
$\Ric(\omega)\geq\kappa\omega$. The trace is negative whenever
$\Area_R(f)>\Acrit$. Selecting one real direction and the sign of its parameter
gives descent; the common parameter radius and Taylor remainder give a
uniform step size under fixed area and weighted-derivative bounds.

\begin{proof}[Proof of Proposition~\ref{cor:compact-fiber-descent}]
Put $x_0=f(0)$ and $\lambda=\lambda_{x_0}^F$.

For the multiplier \eqref{eq:scalar-multiplier}, one has
$N_\Phi=n$, $D_\Phi=\pi(n+1)$, and, by \eqref{eq:scalartrace},
\begin{equation}\label{eq:negative-trace-margin}
 4\pi(n+1)-2\kappa \Area_R(f)\leq-2\kappa\delta.
\end{equation}
Apply Theorem~\ref{thm:engine} with trace error
$\varepsilon_{\mathrm{tr}}=\kappa\delta$, choose the value $\rho_{\mathrm{tr}}$ supplied
by part~\textup{(iii)}, and abbreviate
$f_{a,t}=f^{\rho_{\mathrm{tr}}}_{a,t}$,
$\ell_{a,\tau}=\ell^{\rho_{\mathrm{tr}}}_{a,\tau}$, and
$H_{a,\tau}=H^{\rho_{\mathrm{tr}}}_{a,\tau}$.
For every column $a$ and every
$|t|\leq t_0$, the identities $\Phi_a(0)=0$ and
$d(\widetilde\pi_X)_{\widetilde f(0)}\theta(0)=e(0)$ give the exact
interpolation formula
\begin{equation}
 \begin{aligned}
 f_{a,t}(0)&=f(0),\\
 (f_{a,t})_z(0)&=f_z(0)+t\,e(0)\Phi_a'(0),\\
 \lambda((f_{a,t})_z(0))
 &=\lambda(f_z(0))+t\lambda(e(0)\Phi_a'(0))=1.
 \end{aligned}
 \label{eq:exact-scalar-interpolation}
\end{equation}
Indeed, the ambient perturbation at $z=0$
equals $\widetilde f(0)$ for every $t$, so the chain rule always evaluates
$d\widetilde\pi_X$ at the same point $\widetilde f(0)$.
Moreover, the product rule gives
$(\Xi_{\rho_{\mathrm{tr}}})_a'(0)=\theta(0)\Phi_a'(0)$, so radial regularization does not
alter the derivative of the increment at the origin. Formula
\eqref{eq:exact-scalar-interpolation} fixes the displayed scalar
component, not the entire first jet. For this choice, part~\textup{(iii)}
and \eqref{eq:negative-trace-margin} give
$\sum_a(H_{a,1}+H_{a,i})\leq-\kappa\delta$.
One of the $2n$ real axes therefore has second area derivative
$H_{a,\tau}\leq-\kappa\delta/(2n)$. Choose the sign of $s$ so that
$\ell_{a,\tau}s\leq0$. Then \eqref{eq:Taylor} gives
\[
 \Area_R(f_{a,\tau s})-\Area_R(f)
 \leq-\frac{\kappa\delta}{4n}s^2+C_T|s|^3.
\]
Before selecting the step, decrease the engine parameter radius to
$\min\{t_0,1,d_K/(2\max\{1,C_{\theta,\infty}\})\}$ if necessary,
so that the selected perturbation segment lies in the compact tube
used in Proposition~\ref{eff:prop:joint-estimate}. Set
\begin{equation}\label{eq:descent-step}
 t_*\coloneqq\min\left\{\frac{t_0}{2},
 \frac{\kappa\delta}{8n\max\{C_T,1\}}\right\},
 \qquad
 \sigma\coloneqq\frac{\kappa\delta}{8n}t_*^2.
\end{equation}
Take $s\in\{t_*,-t_*\}$ with the chosen sign and set $g=f_{a,\tau s}$.
The Taylor estimate gives $\Area_R(g)\le\Area_R(f)-\sigma$.
The selected increment has the bounds in \eqref{eq:increment-bound}.
Proposition~\ref{eff:prop:joint-estimate} gives the remaining estimate in
\eqref{eq:quantitative-descent} with
\[
 \beta_I=C_{A,I}t_*/\pi,
\]
where $C_{A,I}$ is \eqref{eff:eq:general-cost} for this increment.
The parameter choice keeps its entire segment in the common tube.
The constants $t_0,C_T,C_{A,I}$ depend only on the fixed caps and target
data. The multiplier has the same cost for every nonzero covector
$\lambda_{x_0}^F$. Thus $t_*,\sigma,\beta_I$ are independent of the
radius, the disk, and the center in $F$.
\end{proof}

For scalar enlargement, it suffices to lift one finite-area disk and
perturb its lift in one constant ambient direction.

\begin{proof}[Proof of Proposition~\ref{prop:finite-area-enlargement}]
\par\smallskip\noindent\textup{(i)}
Suppose that $h$ extends holomorphically past $\overline\D_R$.
Theorem~\ref{thm:uniform-lift}\textup{(ii)} supplies a lift $H$ defined
near $\overline\D_R$, with $H(\overline\D_R)\subset K_{A_0,B_0}$,
$\Pi\circ H=h$, and
$\int_{\D_R}|H'|^2\,dA\leq D_{\widetilde f}$; here
$\Pi=\widetilde\pi_X$ and the compact set, ambient neighborhood, and
Dirichlet bound depend only on the caps and fixed target data.
Enlarge $D=D_{\widetilde f}$ to at least $1$ and put
\[
 Q=\sqrt{D/\pi},\qquad
 \bar t=\min\{1,d_K/(2\max\{1,Q\})\},
\]
where $d_K$ is the common tube radius used in
Section~\ref{eff:subsec:joint-estimate}. For $v=H'(0)$,
Lemma~\ref{eff:lem:weighted-euclidean} gives
\[
 R|v|\le Q,\qquad\|zv\|_\infty\le Q,\qquad
 \int_{\D_R}|(zv)'|^2dA=\pi R^2|v|^2\le D.
\]
The affine formula \eqref{eq:explicit-enlargement} therefore stays in the
common tube and extends past the closed disk.
Proposition~\ref{eff:prop:joint-estimate} gives the area and $B^2$
bounds with
\begin{equation}\label{eff:eq:CII}
 C_{II}=D(L_gQ+3M_0).
\end{equation}
Source dilation satisfies the exact identities
\[
 \Area_{(1+t)R}(G_t)=\Area_R(h_t),\qquad
 B_{(1+t)R}(G_t)=B_R(h_t).
\]
The affine increment vanishes at $0$ and gives
$(h_t)_z(0)=(1+t)h_z(0)$; dilation restores the entire derivative.
The displacement estimate refers to corresponding source points:
\[
 \dist_\omega(G_t((1+t)z),h(z))\le\sqrt{M_0}\,Q t.
\]

\par\smallskip\noindent\textup{(ii)}
Now let $h$ be given only on $\D_R$. For a chosen $t$, the restriction
radius $r=(1+t/2)R/(1+t)$ is strictly smaller than $R$. Hence
$h|_{\D_r}$ extends past its closed disk, and restriction gives
\[
 \Area_r(h)\le\Area_R(h)\le A_0,\qquad
 B_r(h)\le B_R(h)\le B_0.
\]
Apply the closed-disk construction at radius $r$; its dilated map is
exactly $g$ in \eqref{eff:eq:closed-repair} and extends past
$\overline\D_{R'}$. Its two-sided estimates read
\[
 |\Area_{R'}(g)-\Area_r(h)|\le C_{II}t,\qquad
 |B_{R'}(g)^2-B_r(h)^2|\le C_{II}t/\pi.
\]
The restriction inequalities give \eqref{eff:eq:closed-repair-cost},
while the closed-disk interpolation gives $g(0)=h(0)$ and
$g_z(0)=h_z(0)$. The finite displacement bound compares $g((1+t)z)$
with $h(z)$ for $z\in\D_r$; it does not give a closeness estimate for
the principal-map recovery.
\end{proof}

\section{Compactness, marked values, and finite-area disks}
\label{sec:area-descent}

Here we prove the compactness and removal statements used above.
Marked compactness is applied on one fixed interior disk, with all
component areas counted together; a local small-area estimate supplies
the weighted derivative bound for an individual finite-area disk, and
point selection produces the sphere used in recovery.

Throughout this section, $X$ is a compact complex manifold equipped
with a smooth Hermitian metric $\omega$.

\subsection{Finite-area removal}\label{subsec:finite-area-removal}

Finite area fills the punctures of a principal limit, including infinity
when its source is $\C$. We first prove removal so that the
compactness statement can include the extended principal map.

With $g_\omega(v,w)=\omega(v,Jw)$, the energy convention in the compactness
and removal results used below is
\begin{equation}
 \|df\|_{L^2}^2
 =\int\bigl(|f_x|_{g_\omega}^2+|f_y|_{g_\omega}^2\bigr)\,dA
 =2\int f^*\omega=2\Area(f),
 \label{eq:compactness-energy-normalization}
\end{equation}
because $f_y=Jf_x$. Thus their energy hypotheses are equivalent to our
area hypotheses, with the fixed factor $2$.

\begin{proof}[Proof of Theorem~\ref{thm:removal}]
Finite area gives
$\int_{0<|z|<r}f^*\omega\to0$ as $r\downarrow0$, so every sufficiently
small concentric annulus has small energy by
\eqref{eq:compactness-energy-normalization}.
Compactness of $X$ supplies completeness of the target metric and
uniform continuity of its complex structure. Corollary~1.6 of Ivashkovich and
Shevchishin~\cite{IvashkovichShevchishin2000}
therefore gives a continuous extension at the puncture.
In a target chart about the limiting value, the coordinate functions
extend holomorphically by the Riemann removable singularity theorem.
For an entire map, $\zeta=1/z$ identifies infinity with a puncture and
preserves area, so the same argument applies there.
See also~\cite[Theorem~4.2.1]{McDuffSalamon1994}.
\end{proof}

For maps into $\PP^1$, removal also follows from the great Picard
theorem~\cite[\S15, pp.~164--166]{Picard1880}: an essential singularity
would give infinitely many preimages of all but at most two values,
and the area formula with multiplicities would force infinite
spherical area.

\subsection{Principal compactness}
\label{sec:marked-compactness}

Spheres obtained by rescaling near energy concentration points go back
to Sacks and Uhlenbeck~\cite[pp.~2--3 and Section~4]{SacksUhlenbeck1981};
Gromov's compactness theory records how such components
attach~\cite{Gromov1985}. We use the stable convergence of
Definition~2.5 and the compactness theorem with boundary in
Theorem~1 of Ivashkovich and
Shevchishin~\cite{IvashkovichShevchishin2000}.

The \emph{principal map} is the limit in the original source coordinate
and may be constant. A \emph{bubble} is a sphere obtained by rescaling;
the finite-area extension of an entire principal map is also a sphere,
but need not be a bubble.

\begin{proof}[Proof of Proposition~\ref{prop:principal-compactness}]
First pass to a subsequence whose total areas converge to the original
lower limit $A_\infty$.
A diagonal weak selection of the area measures on compact subdisks
gives a Radon measure $\mu$ on $\D_{R_\infty}$.
For every continuous compactly supported $\chi$ with $0\leq\chi\leq1$,
\[
 \int\chi\,d\mu
 =\lim_j\int\chi\,f_j^*\omega\leq A_\infty,
\]
so $\mu(\D_{R_\infty})\leq A_\infty$ by exhaustion.
Fix $p>2$ and an area threshold $\epsilon_0>0$ small enough to apply
Lemma~1.1 and Corollary~1.3 of Ivashkovich and
Shevchishin~\cite{IvashkovichShevchishin2000}, using
\eqref{eq:compactness-energy-normalization}, and set
\[
 \Sigma_\infty=\{z:\mu(\{z\})\geq\epsilon_0\}.
\]
The mass bound makes this set finite.
For $z\notin\Sigma_\infty$, choose a source disk with $\mu$-null
boundary and $\mu$-mass less than $\epsilon_0$.
Weak convergence puts its area below that threshold for all large $j$.
The cited result gives strong local $W^{1,p}$ convergence on a smaller
disk; Sobolev embedding gives uniform convergence, and target charts
and Cauchy estimates give smooth convergence.
A diagonal choice on a countable cover gives $f_\infty$ on the complement
of $\Sigma_\infty$ in the original coordinate. Its local definitions
agree because they are limits of the same subsequence.
Smooth convergence and exhaustion give
$\Area(f_\infty)\leq A_\infty$.
Theorem~\ref{thm:removal} fills the finitely many punctures without
changing area.
\end{proof}

\subsection{Two marked values and incidence}\label{subsec:marked-incidence}

Fixing the complete first jet at $0$ does not ensure that a nonconstant
bubble retains the prescribed tangent direction. In the following
example, the marked value lies on the bubble, while that direction
remains on the principal component. On
the unit disk, with the product Fubini--Study metric, consider
\[
 \widehat f_j(z)\coloneqq \bigl([jz^2:1],[z:1]\bigr)
 :\D\longrightarrow\PP^1\times\PP^1.
\]
The areas are uniformly bounded, and in the affine coordinates
$([w:1],[v:1])$ the complete first jet is fixed:
\[
 \widehat f_j(0)=([0:1],[0:1]),\qquad (\widehat f_j)_z(0)=(0,1).
\]
Away from zero, the principal limit is
\[
 \widehat f_\infty(z)=([1:0],[z:1]).
\]
Its extension has value $([1:0],[0:1])$ at zero, different from the
marked value. Rescaling by $z=\zeta/\sqrt j$ instead gives
\[
 \widehat f_j(\zeta/\sqrt j)
 =\bigl([\zeta^2:1],[\zeta/\sqrt j:1]\bigr)
 \longrightarrow\bigl([\zeta^2:1],[0:1]\bigr).
\]
The prescribed second-factor direction remains on the residual principal
component. Correspondingly,
$B_1(\widehat f_j)/|(\widehat f_j)_z(0)|_\omega\to\infty$.

Marked and nodal values retain incidence.
Lemma~\ref{lem:two-marked-expanding-chain} gives a chain joining two
compact target sets with controlled total area. It supplies the
compactness alternative in Section~\ref{sec:spread-sphere-quotient} and
the final passage to the explicit area bound.

\begin{proof}[Proof of Lemma~\ref{lem:two-marked-expanding-chain}]
First select a subsequence with
$\Area_{R_j}(f_j)\to A_\infty=\liminf_j\Area_{R_j}(f_j)$.
Proposition~\ref{prop:principal-compactness} gives an entire principal
map $f_\infty$ and a finite concentration set $\Sigma_\infty$ on a
further subsequence.

Choose $1<r<\rho<r'$ with $\Sigma_\infty\subset\D_r$.
To retain the source coordinate in the limit, apply compactness to the
graphs
\[
 \widehat f_j(z)=(z,f_j(z)):\D_\rho\longrightarrow\PP^1\times X.
\]
The product Hermitian metric gives a uniform area bound: the first
factor adds only the fixed Fubini--Study area of $\D_\rho$. The maps
are defined on $\D_{r'}$ for large $j$ and converge smoothly on a
collar of $\partial\D_\rho$. Since the source complex structures are
fixed, Theorem~1 of Ivashkovich and
Shevchishin~\cite{IvashkovichShevchishin2000} applies, and its
boundary-collar conclusion allows the parameterizations to agree with
the original ones on a smaller collar.

Let $\sigma_j:\Sigma\to\overline\D_\rho$ and
$\sigma_\infty:\Sigma\to C_\infty$ be the parameterizations in
Definition~2.5, and write
$(p,v):C_\infty\to\PP^1\times X$ for the limiting nodal map.
For each $a\in\{0,1\}$, choose $\xi_{j,a}\in\Sigma$ such that
$\sigma_j(\xi_{j,a})=a$. After passing to a further subsequence,
$\xi_{j,a}\to\xi_a$ for $a=0,1$; set
$q_a=\sigma_\infty(\xi_a)$. The uniform convergence in
Definition~2.5 gives
\[
 (p,v)(q_a)
 =\lim_{j\to\infty}
   (\widehat f_j\circ\sigma_j)(\xi_{j,a})
 =\lim_{j\to\infty}(a,f_j(a))
 \in\{a\}\times K_a.
\]
Thus $p(q_a)=a$ and $v(q_a)\in K_a$ for $a=0,1$.

The genus-zero domain has one bordered root $C_0$ and a finite tree
of spheres. Since $p$ takes values in
$\overline\D_\rho\subset\C$, it is constant on every sphere.
The unchanged outer collar gives $p(z)=z$ there. The argument principle
on $C_0$ consequently gives exactly one preimage, counted with
multiplicity, for every point of $\D_\rho$. Thus
$p:C_0\to\D_\rho$ is biholomorphic. In this coordinate, convergence
off $\Sigma_\infty$ identifies $v|_{C_0}$ with
$f_\infty|_{\D_\rho}$; holomorphic uniqueness extends the identification
across the omitted points.

Choose compact regions in the interior of the root and in each
nonconstant sphere, avoiding the nodes. A single stable
parametrization gives orientation-preserving embeddings of these
regions into $\D_\rho$ with pairwise disjoint images. Uniform and
strong $W^{1,2}$ convergence give convergence of their
$\omega$-integrals.
For any fixed $L>\rho$, these regions are also disjoint from the
annulus $\D_L\setminus\overline\D_\rho$, where the original maps
converge smoothly to $f_\infty$, since $\Sigma_\infty\subset\D_r$.
All these regions lie in $\D_{R_j}$ for large $j$, so their limiting
areas sum to at most $A_\infty$. Exhaust the root interior and the
spheres away from the nodes, then let $L\to\infty$. This gives
\begin{equation}\label{eq:global-root-component-area}
 \Area(f_\infty)+\sum_\alpha\Area(u_\alpha)\le A_\infty,
\end{equation}
where $u_\alpha$ are the nonconstant sphere maps.
Only the original form $\omega$ is integrated, with mapping
multiplicities; the auxiliary Fubini--Study area does not enter.

Theorem~\ref{thm:removal} extends $f_\infty$ across infinity.
Replace the root by this sphere. For each $a\in\{0,1\}$, choose a
component containing $q_a$, and take the path between these two
components in the dual tree. Its endpoint images meet $K_0$ and $K_1$,
respectively. Omit constant components along the path; equality at the
nodes preserves all intersections and endpoint incidences. At least
one nonconstant component remains because $K_0\cap K_1=\varnothing$.
Discarding the other components preserves the upper area bound and
gives the required chain.
\end{proof}

\subsection{Rescaling and finite-area weighted derivatives}
\label{subsec:weighted-derivative-proof}

Lemma~\ref{lem:pointselect} is the standard affine form of Brody's
reparametrization lemma~\cite[Lemma~2.1]{Brody1978}; see
also~\cite[Section~1]{Duval2021}. We include the proof to record the
normalization and the finite-area estimate used in the recovery argument.
For a single finite-area disk, a local small-area estimate proves
Lemma~\ref{lem:finite-area-weighted-bound} directly.

\begin{lemma}\label{lem:pointselect}
Let $(X,\omega)$ be a compact Hermitian manifold, and let $R_j>0$.
Let $f_j:U_j\to X$ be holomorphic maps, where $U_j\subset\C$ is open and
$\overline\D_{R_j}\Subset U_j$, and suppose
$B_{R_j}(f_j)\to\infty$.  Then after affine source rescaling, a subsequence
converges locally smoothly on $\C$ to a nonconstant entire map
$u:\C\to X$.  If in addition $\sup_j\Area_{R_j}(f_j)<\infty$, then
\[
 \Area(u)\leq\liminf_j\Area_{R_j}(f_j).
\]
\end{lemma}

\begin{proof}
Under the additional area hypothesis, first pass to a subsequence for which
$\Area_{R_j}(f_j)$ converges to the lower limit of the original area sequence.
Because $f_j$ extends holomorphically past $\overline\D_{R_j}$, the
function
\[
 z\longmapsto (R_j-|z|)|(f_j)_z(z)|_\omega
\]
extends continuously to $\overline\D_{R_j}$ by taking the value zero
on the boundary. It therefore attains its maximum. For all large $j$
this maximum is positive, so a maximizing point $w_j$ lies in
$\D_{R_j}$. Set
$d_j\coloneqq R_j-|w_j|$ and
$\Lambda_j\coloneqq |(f_j)_z(w_j)|_\omega$.  Then
\[
 d_j\Lambda_j=B_{R_j}(f_j)\longrightarrow\infty,\qquad
 |(f_j)_z(z)|\leq2\Lambda_j\quad(|z-w_j|<d_j/2).
\]
Indeed, $|z-w_j|<d_j/2$ implies $R_j-|z|>d_j/2$.
Maximality then gives
$(d_j/2)|(f_j)_z(z)|\leq d_j\Lambda_j$, which is the second bound.
The rescaled map
\[
 u_j(\zeta)\coloneqq f_j(w_j+\zeta/\Lambda_j),
 \qquad |\zeta|<\frac12d_j\Lambda_j,
\]
satisfies $|(u_j)_\zeta|_\omega\leq2$ and
$|(u_j)_\zeta(0)|_\omega=1$. The source radii
$d_j\Lambda_j/2$ tend to infinity. On every fixed compact source disk,
the derivative bound gives equicontinuity, and compactness of $X$
gives a uniformly convergent subsequence. A diagonal selection and
Cauchy estimates in local target charts give a smooth limit
$u:\C\to X$ with $|u_\zeta(0)|_\omega=1$; hence $u$ is nonconstant.
Under the additional area hypothesis, conformal invariance gives, for
each fixed $S>0$,
\[
 \Area_S(u)=\lim_j\Area_S(u_j)\leq\lim_j\Area_{R_j}(f_j)
 =\liminf_j\Area_{R_j}(f_j).
\]
Letting $S\to\infty$ proves the area bound.
\end{proof}

\begin{proof}[Proof of Lemma~\ref{lem:finite-area-weighted-bound}]
There are constants $\epsilon,c>0$, depending only on $(X,\omega)$,
such that every holomorphic $u:\D\to X$ of area less than $\epsilon$
satisfies $|u_z(0)|_\omega^2\leq c\Area_1(u)$.
Indeed, Lemma~1.1 of Ivashkovich and
Shevchishin~\cite{IvashkovichShevchishin2000}, applied with $p=4$
and the energy convention
\eqref{eq:compactness-energy-normalization}, bounds
$\|du\|_{L^4(\D_{1/2})}$ by a constant times
$\Area_1(u)^{1/2}$. Sobolev--Morrey, applied
after a fixed smooth embedding of $X$ into Euclidean space, then
gives a small image diameter. After decreasing $\epsilon$, the image
of $\D_{1/2}$ lies in one uniformly
controlled holomorphic chart. The submean inequality for its
holomorphic coordinate derivatives gives the stated estimate.

Choose $r<R$ so that the area of $h$ on
$\{r<|w|<R\}$ is less than $\epsilon$.
For $|z|>(2r+R)/3$, the disk of radius
$a=(R-|z|)/2$ about $z$ lies in this annulus. Dilation of the
local estimate gives
\[
 (R-|z|)^2|h_z(z)|_\omega^2
 \leq4c\int_{\D_a(z)}h^*\omega\leq4c\epsilon.
\]
On the remaining compact subdisk, $h_z$ is bounded.
Thus $B_R(h)<\infty$.
\end{proof}

\section{Applications}
\label{sec:applications}

The relative theorem recovers the rational connectedness of Fano
manifolds \cite{Campana1992,KMMFano1992}: two general points can be
joined by a rational curve. Choose a positive-Ricci metric as
in Proposition~\ref{prop:fano-metric}. For a quasifibration
$q:X\dashrightarrow Y$ in Campana's sense~\cite[0.7]{Campana1992}, with
$\dim Y>0$, restrict to the dense open subsets on which it is a proper
morphism with connected fibers. Generic smoothness~\cite[III, Corollary~10.7]{HartshorneBook} makes its general fiber
$F$ smooth and compact. Theorem~\ref{thm:bounded-fiber-escape} gives a
rational curve meeting $F$ but not contained in it. Campana's criterion~\cite[Proposition~2.7]{Campana1992} gives rational chain connectedness.
For smooth projective varieties over $\C$, this implies rational
connectedness~\cite{AraujoKollar2002}.

The pointwise construction also recovers Mori's characterization of
projective space by an ample tangent bundle~\cite{Mori1979}; for an
earlier characteristic-zero proof, see Peternell~\cite[p.~99 and Section~2]{Peternell1996}.
Let $X$ be a connected smooth complex projective $n$-fold with $T_X$
ample. The case $n=1$ is immediate. For $n\ge2$, the determinant
$K_X^{-1}=\det T_X$ is ample~\cite{Hartshorne1966}, so $X$ is Fano and
Theorem~\ref{thm:intro-pointwise-fano} supplies a rational curve.
Starting from this curve, two-point degree reduction gives a rational
curve $C$ with
$-K_X\cdot C\leq n+1$~\cite[Section~5, Theorem~3, first step]{Demazure1981}.
Let $\nu:\PP^1\to C\subset X$ be its normalization. Since $T_X$ is
ample,
\[
 \nu^*T_X\simeq\bigoplus_{i=1}^n\cO_{\PP^1}(a_i),
 \qquad a_i\geq1.
\]
The nonzero differential
\[
 d\nu:\cO_{\PP^1}(2)\longrightarrow \nu^*T_X
\]
forces some $a_i\geq2$. Hence
\[
 -K_X\cdot C=\deg\nu^*T_X=\sum_{i=1}^n a_i\geq n+1.
\]
Thus equality holds. The resulting minimal-curve properties are those
in~\cite[Section~5, Corollary~(a),(c)]{Demazure1981}.
Mori's classification using the resulting family gives $X\simeq\PP^n$~\cite{Mori1979}; see also~\cite[Section~6]{Demazure1981}.
These steps take place over $\C$; the analytic construction supplies
the initial rational curve.

\section*{Acknowledgments}
We are grateful to Franc Forstneri\v{c} for introducing us to Oka geometry
and for his encouragement. We also thank Jun-Muk Hwang and
Mihai P\u{a}un for their helpful suggestions.

\section*{Funding}

S.-Y.~Xie was partially supported by the National Key R\&D Program of China (grant nos.~2023YFA1010500 and 2021YFA1003100), the National Natural Science Foundation of China (grant nos.~12288201 and 12471081), and the Xiaomi Young Talents Program.

\section*{AI use disclosure}
The mathematical ideas and overall approach are due to the authors. AI tools
assisted with exploratory work and language editing during the development
of the proofs. The authors reviewed and approved every mathematical statement
and take full responsibility for the paper's content.

\end{document}